\documentclass[12pt,english,a4paper]{smfart}

\usepackage[T1]{fontenc}
\usepackage{lmodern}
\usepackage{smfthm, dsfont,xfrac}
\usepackage[headings]{fullpage}
\usepackage{amssymb, mathrsfs, hyperref,stmaryrd, color, graphicx, array, verbatim}
\usepackage[all]{hypcap}    
\usepackage[dvipsnames]{xcolor}

\usepackage{tikz}

\usetikzlibrary {positioning}
\usetikzlibrary {arrows.meta} 
\usetikzlibrary {graphs} 
\usetikzlibrary{matrix}
\usetikzlibrary {graphs.standard} 
\usetikzlibrary{backgrounds}
\usetikzlibrary{calc}
\usetikzlibrary{patterns}
\renewcommand{\geq}{\geqslant}
\renewcommand{\leq}{\leqslant} 
\renewcommand{\ge}{\geqslant}
\renewcommand{\le}{\leqslant} 

\newcommand{\defini}{\textbf}

\newcommand{\Sm}{\mathscr{S}}
\newcommand{\La}{\mathscr{L}}

\NumberTheoremsIn{section}

\definecolor{section}{HTML}{e6308a} 
\definecolor{lift}{HTML}{b51963}
\definecolor{base}{HTML}{0073e6}
\definecolor{u3}{HTML}{ff8611}
\definecolor{u2}{HTML}{ede529}
\definecolor{u4}{HTML}{1bde54}
\definecolor{u1}{HTML}{1de4a9}
\definecolor{redgentil}{HTML}{b51963}

\date{March 2026}

\author{S\'ebastien Martineau}
\address{LPSM, Sorbonne Universit\'e\\
 4 Place Jussieu, Case 158\\
 75252 Paris Cedex 05}
\email{smartineau@lpsm.paris}
\urladdr{\href{https://martineau-maths.github.io/webpage/}{martineau-maths.github.io/webpage}}

\author{R\'emy Poudevigne-{}-Auboiron\thanks{R\'emy Poudevigne-{}-Auboiron has been supported by the project ANR LOCAL (ANR-22-CE40-0012-02) operated by the Agence Nationale de la Recherche (ANR)}}
\address{LPSM, Sorbonne Universit\'e\\
 4 Place Jussieu, Case 158\\
 75252 Paris Cedex 05}
\email{rpoudevigne@lpsm.paris}
\urladdr{}

\author{Paul Rax}
\address{MathNet Team, INRIA Paris\\
48 Rue Barrault \\
75013 Paris 13}
\email{paul-pierre.rax@inria.fr}
\urladdr{\href{https://paul-rax.github.io/webpage/}{paul-rax.github.io/webpage}}

\title[Stochastic domination and lifts of random variables in percolation]{Stochastic domination and lifts of random variables in percolation theory}

\begin{abstract}
    Consider some matrix waiting for its coefficients to be written. For each column, sample independently a Bernoulli random variable of some parameter $p$. Seeing all this and possibly using extra randomness, Alice then chooses one spot in each column, in any way she wants. When the Bernoulli random variable of some column is equal to 1, the number 1 is written in the chosen spot. When the Bernoulli random variable of a column is 0, nothing is done on this column. We prove that, using extra randomness, it is possible for Bob to fill the empty entries with well chosen 0's and 1's so that the entries of the matrix are independent Bernoulli random variables of parameter $p$. We investigate various generalisations and variations of this problem, and use this result to revisit and generalise (nonstrict) monotonicity of the percolation threshold $p_c$ with respect to a form of graph-quotienting, namely fibrations. We also use this result to revisit the BK inequality.
    
    In a second part, which is independent of the first one, we revisit strict monotonicity of $p_c$ with respect to fibrations, a result that naturally requires more assumptions than its nonstrict counterpart. We reprove the bond-percolation case of the result of Martineau--Severo without resorting to essential enhancements, using couplings instead.
\end{abstract}

\begin{document}

\maketitle
\setcounter{tocdepth}{2}

\section{Introduction}

\subsection{Stochastic domination of lifts} One purpose of this article is to prove the following theorem and study several generalisations of it. Throughout the text, the notation $\pi^{-1}(b)$ stands for the set $\pi^{-1}(\{b\})$. Theorem~\ref{lakon} is best read with an eye on Figure~\ref{fig:lakon}.

\begin{theo}
\label{lakon}
Let $\pi : A\to B$ be a surjective map between nonempty countable sets. Let $p\in [0,1]$. Let $(X_b)_{b\in B}$ be a collection of independent random variables with Bernoulli distribution of parameter $p$. On the same probability space, let $(S(b))_{b\in B}$ be a collection of random variables such that, for every $b\in B$, the random variable $S(b)$ takes values in $\pi^{-1}(b)$.
At last, for every $a\in A$, set $Y_a$ to be $X_{\pi(a)}$ if $S\circ\pi(a)=a$ and 0 otherwise. 

Then, the distribution of $(Y_a)_{a\in A}$ is stochastically dominated by $\mathrm{Bernoulli}(p)^{\otimes A}$.
\end{theo}

\begin{center}
\begin{figure}
    \centering
    \begin{tikzpicture}
\matrix(base)[nodes={minimum size=6mm}, column sep = 1pt, row sep = 1pt]
  {
     &[10pt] & & & & &\node(indics)[fill=white,draw = black] {};&  \\    
     & & \node[fill=section,draw = black]{};& & & \node[fill=white,draw = black]{};& \node[fill=section,draw = black] {};&  \\    
     \node(a)[color = black, font=\large]{$A$};& \node[fill=white,draw = black]{};& \node[fill=white,draw = black]{};& &\node[fill=section,draw = black]{}; & \node[fill=section,draw = black]{};& \node[fill=white,draw = black] {};&\node[fill=white,draw = black]{};  \\ 
     & \node[fill=section,draw = black]{};& \node[fill=white,draw = black]{};& \node[fill=white,draw = black]{};&\node[fill=white,draw = black]{}; & \node[fill=white,draw = black]{};& \node[fill=white,draw = black] {};&\node[fill=section,draw = black]{};  \\ 
     & \node[fill=white,draw = black]{};& \node[fill=white,draw = black]{};& \node[fill=section,draw = black]{};&\node[fill=white,draw = black]{}; & \node[fill=white,draw = black]{};& \node[fill=white,draw = black] {};&\node[fill=white,draw = black]{};  \\[15pt]
\node(b)[color = black, font=\large]{$B$};& \node[color = base, draw = black]{$1$};& \node[color = base, draw = black]{$0$};& \node[color = base, draw = black]{$1$};& \node(indicx)[color = base, draw = black]{$0$};& \node[color = base, draw = black]{$0$};& \node[color = base, draw = black]{$1$};& \node(indicf)[color = base, draw = black]{$1$};  \\ 
  };
  \node[above left=10 pt of indics, font=\large, color = section]{$S(b)$};
  \node[color = base, below = 10pt of indicx, font=\large]{$X_b$};
  \draw [->,thick] (a) -- (b) node [left,midway, font=\large] {$\pi$};
  \matrix(final)[right =100pt of base, nodes={minimum size=6mm, draw = black}, column sep = 1pt, row sep = 1pt]
  {
       & & & & &\node(indicy){$0$};&  \\    
     & \node[color = lift,fill=lightgray!20]{$0$};& & & \node{$0$};& \node[color = lift,fill=lightgray!20]{$1$};&  \\    
     \node{$0$};& \node{$0$};& &\node[color = lift,fill=lightgray!20]{$0$}; & \node[color = lift,fill=lightgray!20]{$0$};& \node{$0$};&\node{$0$};  \\ 
    \node[color = lift,fill=lightgray!20]{$1$};& \node{$0$};& \node{$0$};&\node{$0$}; & \node{$0$};& \node{$0$};&\node[color = lift,fill=lightgray!20]{$1$};  \\ 
     \node{$0$};& \node{$0$};& \node[color = lift,fill=lightgray!20]{$1$};&\node{$0$}; & \node{$0$};& \node{$0$};&\node{$0$};  \\ 
  };
      \node[above left=10 pt of indicy, font=\large, color = lift]{$Y_a$};
      \draw [-{Stealth[length=5mm]},shorten >=15pt,shorten <=15pt,thick, color = lift] (base) -- (final);
\end{tikzpicture}

    \caption{Illustration of the notation of Theorem~\ref{lakon}.}
    \label{fig:lakon}
\end{figure}
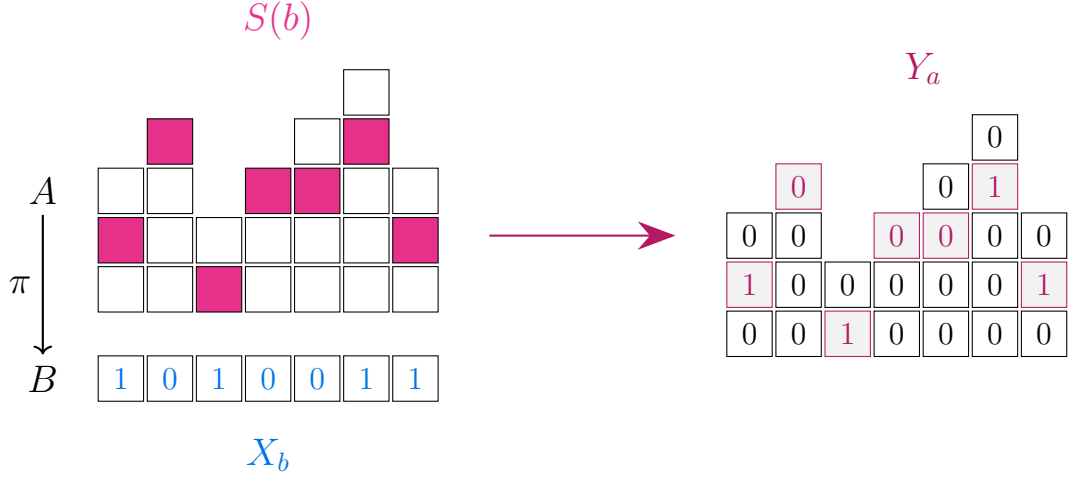
\end{center}

The conclusion of Theorem~\ref{lakon} means that it is possible to find a new probability space and random variables $(Y'_a)_{a\in A}$ and $(Z_a)_{a\in A}$ defined on this new probability space such that the following conditions hold:
\begin{enumerate}
\item the joint distribution of $(Y'_a)_{a\in A}$ is equal to that of $(Y_a)_{a\in A}$,
\item the random variables $Z_a$ are i.i.d. with Bernoulli distribution of parameter $p$,
\item and for every $a\in A$, we almost surely have $Y'_a\le Z_a$.
\end{enumerate}

\begin{rema}
In Theorem~\ref{lakon}, the random variables $S(b)$ are allowed to be highly correlated, both relative to each other and relative to $(X_b)$. It is good to have in mind the case when every $S(b)$ is of the form $f_b\left((X_{b'})_{b'\in B}\right)$ for some measurable function $f_b$.
\end{rema}

What drew us to Theorem~\ref{lakon} lies in the realm of percolation theory. In Section~\ref{sec:genperco}, we provide general context, definitions and setup regarding percolation theory. Connections between percolation theory and Theorem~\ref{lakon} are then explained in Section~\ref{sec:percolakon}.

\subsection{General background regarding percolation theory}\label{sec:genperco}
Let $\mathscr G$ be some countable locally finite graph\footnote{By ``graph'', we mean a simple graph, i.e.~the edges are undirected and we do not allow self-loops or multiple edges. ``Locally finite'' means that each vertex has a finite number of neighbours. A graph is ``countable'' if its vertex-set is countable --- which entails countability of its edge-set. We take ``countable'' to mean ``finite or countable''. Being interested in infinite connected components, the finite case, although not excluded, will have no interest: finite graphs have critical parameters equal to 1.}. \defini{Site percolation} (resp. \defini{bond percolation}) of parameter $p$ consists in declaring each vertex (resp. edge) to be open with probability $p$, independently. It can be thought of as a random subgraph as follows. In the site case, the new set of vertices is that of open vertices, and two vertices are said to be adjacent if they were adjacent in the original graph. In the bond case, all vertices are kept but only open edges are retained.
In both cases, connected components of this random subgraph are called \defini{clusters}. Percolation theory is devoted to understanding these clusters. In particular, one is interested in the \emph{infinite} clusters: is there any? if so, how many are there ? what is their geometry?

The probability that there is an infinite cluster is a function of $p$ which is weakly increasing and can only take values in $\{0,1\}$, due to Kolmogorov's zero-one law. There is therefore a unique parameter $p_c(\mathscr{G})\in[0,1]$ such that, for every $p\in [0,1]$, this probability is zero if $p<p_c(\mathscr{G})$ and one if $p>p_c(\mathscr{G})$. The number $p_c(\mathscr{G})$ is called the {\it critical parameter} of $\mathscr G$. More precisely, depending on whether one works with site or bond percolation, each graph gives rise to two critical parameters, namely $p_c^\mathrm{site}(\mathscr{G})$ and $p_c^\mathrm{bond}(\mathscr{G})$.

For most graphs, the critical parameters $p_c^\mathrm{site}(\mathscr{G})$ and $p_c^\mathrm{bond}(\mathscr{G})$ are nontrivial, meaning that they are neither equal to 0 nor to 1. Whenever these values are indeed nondegenerate, percolation gives rise to two regimes (one below $p_c$ and one above), and we say that this process undergoes a phase transition.

Percolation has been extensively studied since its introduction by Broadbent and Hammersley almost seventy years ago \cite{bh}, particularly on the $d$-dimensional cubic lattice $\mathbb{Z}^d$. It was introduced as a model for propagation of a fluid in a porous medium, random removal accounting for porosity and clusters for propagation. Since then , percolation has continued to attract attention for several reasons: it is interesting as a model of statistical mechanics; it is connected to diverse mathematical domains; and it is charming and challenging in itself. See \cite{dc} and references therein for a survey.

Since the seminal papers \cite{bs,blyps}, percolation on rather general graphs has been investigated in a systematic way, revealing a deep interplay between percolation, geometry, and group theory. A fair amount of what we revisit in this article lies in this lineage.

\subsection{Revisiting a percolation result via Theorem~\ref{lakon}}\label{sec:percolakon} It was proved in \cite[Theorem~1]{bs} --- see also  \cite[Proposition~4.1]{cmt} --- that if a graph $\Sm$ is a ``quotient'' of  a graph $\La$, then, for both site and bond percolation, we have $p_c(\Sm)\ge p_c(\La)$. Here is a more precise statement, which encompasses quotients but also other situations, via the general concept of fibration.

\begin{enonce}[remark]{Notation}
    We denote the large graph by the symbol $\La$, which one may read ``large'' rather than ``L''. Likewise, we denote the small graph by the symbol $\Sm$, which one may read ``small'' rather than ``S''. Given any graph $\mathscr{G}$, we denote by $V_\mathscr{G}$ its set of vertices and by $E_\mathscr{G}$ its set of edges.
\end{enonce}

In this paper, given two countable locally finite graphs $\La$ and $\Sm$, a \defini{fibration} from $\La$ to $\Sm$ is a surjective map $\pi:V_\La\to V_\Sm$ such that for every vertex $x$ of $\La$ and every neighbour $v$ of $\pi(x)$ in $\Sm$, there is a neighbour $y$ of $x$ in $\La$ such that $\pi(y)=v$.

\begin{rema}
Apart from surjectivity, the definition of fibration states that for every vertex $u$ in $\Sm$ and every neighbour $v$ of $u$, we are able to ``lift'' the edge $\{u,v\}$ in $\Sm$ in order to put the endpoint corresponding to $u$ \emph{anywhere} in $\pi^{-1}(u)$.
\end{rema}

\begin{theo}[Benjamini--Schramm 96]
\label{thm:bs}
	Let $\La$ and $\Sm$ be two countable locally finite graphs. Assume that there is a fibration from $\La$ to $\Sm$.

Then, we have $p_c^\mathrm{site}(\La)\le p_c^\mathrm{site}(\Sm)$ and $p_c^\mathrm{bond}(\La)\le p_c^\mathrm{bond}(\Sm)$.
\end{theo}

The proof of Theorem~\ref{thm:bs} proceeds by exploring (a spanning tree of) the cluster of the origin step by step. Doing so properly ensures that vertices (or edges) yet to be revealed are still, conditionally on previous steps, open independently and with probability $p$. As $\pi$ can lift edges, one is able to lift trees as well. By performing $p$-percolation on $\Sm$, lifting what one sees in the exploration (both closed and open vertices or edges) to $\La$, and then putting i.i.d. $\mathrm{Bernoulli}(p)$ random variables on what remains of $\La$, one gets a coupling such that whenever the origin is connected to infinity in $\Sm$, this is the case as well in $\La$.

This raises the question of whether it is necessary to proceed via such a cautious exploration algorithm, and to lift both positive and negative information. If $p>p_c(\Sm)$, would it not be easier to simply pick an infinite open path in $\Sm$ and lift it to $\La$, thus proving that $p\ge p_c(\La)$? If we lift an infinite path and put i.i.d. $\mathrm{Bernoulli}(p)$ random variables on its complement, it will not work: some vertices will be open with probability larger than $p$, and independence would be problematic as well. Nevertheless, Theorem~\ref{lakon} ensures that it is possible to fill the complement of the lifted path with well chosen random variables so that the final result is $p$-percolation on $\La$.

\begin{proof}[Sketch of proof of Theorem~\ref{thm:bs} using Theorem~\ref{lakon}]
We treat the case of site percolation, as the case of bond percolation can either be tackled in the same way or reduced to that of site percolation --- see Section~\ref{sec:bondtosite}.

Let $p$ be such that there is an infinite open path in $\Sm$ almost surely. Let us show that this is also the case for $\La$. We set $A=V_\La$ and $B=V_\Sm$. 
Let $(X_b)_{b\in B}$ be a $p$-percolation on $\Sm$, where $X_b=1$ means that the vertex $b$ is open and $X_b=0$ means that it is closed. Almost surely, there is an infinite open path: pick one and call it $\kappa$. One can lift it via $\pi$ to some path in $\La$, which we call $\tilde{\kappa}$. For each $b\in B$, if $b$ belongs to $\kappa$, set $S(b)$ to be its corresponding vertex in $\tilde\kappa$, otherwise define $S(b)$ in an arbitrary way. Define $(Y_a)$ as in Theorem~\ref{lakon}. Notice that the configuration $(Y_a)$ contains an infinite open path, namely $\tilde{\kappa}$.
Theorem~\ref{lakon} now states that $(Y_a)$ is stochastically dominated by $p$-percolation on $\La$. As there is almost surely an infinite open path in $(Y_a)$, so is the case for $p$-percolation on $\La$. This completes the proof, except that we did not take care of measurability issues regarding $\kappa$ and $\tilde{\kappa}$. We defer these details to Section~\ref{sec:measurability}.
\end{proof}

\subsection{Other results we revisit} Still in the seminal paper \cite{bs}, Benjamini and Schramm asked a precise question which can be very roughly summarised as: is $p_c$ strictly increasing under strict quotient? The statement they asked for was proved in \cite{ms} by using essential enhancements.\footnote{For a reference on essential enhancements, see the original paper \cite{ag}. As some lemma in this paper is not properly established, see \cite{bbr} for a discussion of what exactly is rigourously known.} We provide a new proof of the bond-percolation case of this result. Our new proof does not use essential enhancements or differential inequalities but couplings instead. This new proof does not rely on Theorem~\ref{lakon} --- it is more traditional and relies on appropriate exploration algorithms.

At last, let us mention that even though we did not investigate stochastic domination of lifts with the BK inequality in mind, our Theorem~\ref{lakon} is tailored to revisit it. This is performed in the short Section~\ref{sec:bk}. The BK inequality itself is presented there.

\subsection*{Structure of the paper} In Section~\ref{sec:main}, we state and prove our main theorem, of which Theorem~\ref{lakon} is a particular case. Section~\ref{sec:gen} is devoted to investigating variations of the main theorem: a rather delicate statement appears when we try to lift several random variables per column --- see Theorem~\ref{thm:multilift}. In Section~\ref{sec:perco}, we apply our main theorem to not only give a new proof of Theorem~\ref{thm:bs} but also generalise it. The BK inequality is revisited in Section~\ref{sec:bk}. At last, we revisit strict monotonicity of $p_c$ under quotient. This is performed in Section~\ref{sec:strict}, which is completely independent of the other sections. All sections can be read mostly independently of each other, except that we advise to read Section~\ref{sec:statement} before Section~\ref{sec:gen}.

\subsection*{Acknowledgements} This project originated as a side project from an internship of the third author under the supervision of the first one and David Garc\'ia-Zelada, whom we thank for the very nice working experience as well as for his clear view of geometry. We are grateful to Fran\c{c}ois Ledrappier and Florent Martineau for stimulating conversations. We are also indebted to David Garc\'ia-Zelada, Thierry L\'evy and Romain Tessera for leading us to understand that the terminology ``weak covering map'' from Martineau--Severo was poorly chosen: these maps are much closer to fibrations in homotopy theory than to covering maps. We are grateful to Franco Severo for asking whether our techniques, then used to prove non-strict inequalities for $p_c$, could be used to recover the strict monotonicity theorem of \cite{ms}: we did not manage to do so but our attempts have led to the results of Section~\ref{sec:multi} and Section~\ref{sec:strict}. We thank Vibhu Dalal for proofreading this article, and the reviewer for their feedback. At last, SM has a warm thought for Vincent Tassion, as Section~\ref{sec:strict} has a flavour very close to what we were trying to do at the very beginning of our doctorate.

\section{Main theorem}
\label{sec:main}

In this section, we establish a vastly generalised version of Theorem~\ref{lakon}, namely Theorem~\ref{main}. The main point of Theorem~\ref{main} is to yield Corollary~\ref{coro:indep}, which is similar to Theorem~\ref{lakon} but with $[0,\infty]$-valued random variables, still taken to be independent\footnote{only between $\pi$-fibres, not within them} but not necessarily identically distributed.
This corollary will be stated using a different framework. Stated in the same way as Theorem~\ref{lakon}, it goes as follows.

For completeness, we first recall what we mean in general by stochastic domination. Let $(E,\mathscr{E})$ be a measurable space. We further assume that $E$ is endowed with a partial order $\leq$ such that $\left\{(x,y)\in E^2\,:\,x\le y\right\}$ belongs to the $\sigma$-field $\mathscr{E}\otimes\mathscr{E}$. Let $\mu$ and $\nu$ be two probability measures on $(E,\mathscr{E})$. We say that $\mu$ is \defini{stochastically dominated} by $\nu$ if there are random variables $X$ and $Y$ defined on the same probability space such that the following three conditions hold:
\begin{enumerate}
    \item the distribution of $X$ is $\mu$,
    \item the distribution of $Y$ is $\nu$,
    \item the inequality $X\le Y$ holds almost surely.
\end{enumerate}
When working on $[0,\infty]^A$, the implicit order will always be the product order, given by
\[
x\le y \quad \iff \quad \forall a\in A,\quad x_a\le y_a.
\]
The set $[0,\infty]$ is endowed with its usual topology, making it a compact space. The space $[0,\infty]^A$, endowed with the product topology, is then compact as well.

\begin{coro}[reformulation of Corollary~\ref{coro:indep}]
\label{coro:indep-va}
    Recall the notations of Theorem \ref{thm:bs}. Then, for every $a\in A$, set $Y_a:=X_{\pi(a)}\,\mathds{1}_{S\circ\pi(a)=a}$.

    For every $b\in B$, let $\rho_b$ be a probability measure on $[0,\infty]^{\pi^{-1}(b)}$. Assume that for every $b\in B$ and every $a\in\pi^{-1}(b)$, the distribution of $X_{b}$ is stochastically dominated by the $a$-marginal of $\rho_{b}$.
    Then, the distribution of $(Y_a)_{a\in A}$ is stochastically dominated by $\bigotimes_{b\in B}\rho_b$.
\end{coro}

\begin{rema}
    \label{rem:reduction}
    Even though this statement seems much more general than Theorem~\ref{lakon}, it can actually be easily deduced from the $p=\sfrac12$ case of Theorem~\ref{lakon}. Indeed, by the trick of the pseudoinverse\footnote{More formally, this is a left adjoint in the sense of Galois connections.} cumulative distribution function, it suffices to deal with the case where each $X_b$ is uniformly distributed on $[0,1]$. Instead of assuming marginals of $\rho_b$ to stochastically dominate $\mathrm{Unif}([0,1])$, we can assume that we are in the least favourable case, namely when these marginals are equal to the uniform distribution on $[0,1]$. By considering binary digits, which are then independent Bernoulli random variables of parameter $\sfrac12$, up to multiplying $A$ and $B$ by $\mathbb{N}$, the general case is reduced to the balanced case of Theorem~\ref{lakon}.
\end{rema}

\subsection{Statement}\label{sec:statement} Let $\pi :A\to B$ be a surjective map between nonempty countable sets. We will be interested in a specific kind of probability measures on $[0,\infty]^A$. Say that a probability measure $\mu$ on $[0,\infty]^A$ is a \defini{$\pi$-lift} if, for $\mu$-almost surely $x$, for every $b\in B$, at most one $a\in \pi^{-1}(b)$ satisfies $x_a\neq 0$.

Notice that a probability measure is a $\pi$-lift if and only if it can be obtained as the distribution of a random variable $Y$ built as follows:
\begin{enumerate}
    \item sample some random variable $X$ taking values in $[0,\infty]^B$,
    \item on the same probability space as $X$, for every $b$, sample some random variable $S(b)$ taking values in $\pi^{-1}(b)$,
    \item for every $a\in A$, set $Y_a=X_{\pi(a)}\,\mathds{1}_{S\circ\pi(a)=a}$.
\end{enumerate}

One purpose of the present paper is to study results of the following type: {\it Let $\mu$ and $\rho$ be two probability measures on $[0,\infty]^A$. Assume that $\mu$ is a $\pi$-lift and ``some suitable assumptions''. Then, the probability measure $\mu$ is stochastically dominated by $\rho$.}

Our main result is the following one. Recall that a \defini{section} $s$ of $\pi$ is a map $s:B\to A$ such that $\pi \circ s =\mathrm{id}_B$. Informally, it is a way to select one element in each $\pi^{-1}(b)$.

\begin{theo}
\label{main}
Let $\pi :A\to B$ be a surjective map between nonempty countable sets. Let $\mu$ and $\rho$ be two probability measures on $[0,\infty]^A$. Assume that $\mu$ is a $\pi$-lift. Further, assume it is possible to fix a section $\mathfrak s$ of $\pi$ such that the following two conditions simultaneously hold:
\begin{itemize}
    \item[$\mathbf{A}_\rho$]\label{item:assa} Let $Z$ be a $\rho$-distributed random variable. For every $b\in B$, let $W_b$ denote the random variable $(Z_c\,:\,c\in A\text{ such that }\pi(c)\neq b)$. For every $b$, every $a\in\pi^{-1}(b)$ and every event $\mathrm{H}\in \sigma(W_b)$ of positive probability, the conditional probability distribution of $Z_{\mathfrak s(b)}$ given $\mathrm{H}$ is stochastically dominated by that of $Z_a$ given $\mathrm{H}$.
    \item[$\mathbf{B}_\mu^\rho$]\label{item:assb} Let $Y$ be a $\mu$-distributed random variable. For every $b\in B$, let $X_b:=\max_{a\in \pi^{-1}(b)}Y_a$. Notice that this maximum is almost surely well defined since $\mu$ is a $\pi$-lift. For every $a\in A$, let $Y'_a=X_{\pi(a)}\,\mathds{1}_{\mathfrak s\circ \pi (a)=a}$. The condition we ask is that the distribution of $Y'$ is stochastically dominated by $\rho$.
\end{itemize}
Then, the probability measure $\mu$ is stochastically dominated by $\rho$.
\end{theo}

The assumptions of Theorem~\ref{main} are not meant to be natural {\it a priori} but to capture the essence of the proof --- and we will see in Section~\ref{sec:somesections} that a natural generalisation of it fails. From Theorem~\ref{main}, we deduce the two following more concrete results, which are derived in Section~\ref{sec:coro}.

If $\mu$ is a $\pi$-lift, the pushforward of $\mu$ by $x\longmapsto (b \longmapsto \sup_{a\in\pi^{-1}(b)} x_a)$ is called the \defini{pushdown} of $\mu$. This pushdown can alternatively be described as follows. Let $Y$ be a $\mu$-distributed random variable. For every $b\in B$, let $Z_b:=\max_{a\in \pi^{-1}(b)}Y_a$. Notice that this maximum is almost surely well defined since $\mu$ is a $\pi$-lift. Then, the distribution of $(Z_b)$ is the pushdown of $\mu$.

\begin{coro}
\label{coro:indep}
Let $\pi : A\to B$ be a surjective map between nonempty countable sets. For every $b\in B$, let $\nu_b$ be a probability measure on $[0,\infty]$ and let $\rho_b$ be a probability measure on $[0,\infty]^{\pi^{-1}(b)}$. Assume that, for every $b\in B$ and every $a\in \pi^{-1}(b)$, the probability measure $\nu_{b}$ is stochastically dominated by the $a$-marginal of $\rho_{b}$.
Let $\mu$ be a $\pi$-lift such that its pushdown is stochastically dominated by $\bigotimes_{b\in B}\nu_b$.

Then, the probability measure $\mu$ is stochastically dominated by $\bigotimes_{b\in B}\rho_b$.
\end{coro}

\begin{rema}
    This version of Corollary~\ref{coro:indep} is slightly stronger than Corollary~\ref{coro:indep-va}, as we only assume the pushforward of $\mu$ to be stochastically dominated by $\bigotimes_b\nu_b$ instead of equal to it. It is actually not hard to derive the strong version from the weak one. Indeed, any $\pi$-lift with such a dominated pushforward measure is stochastically dominated by a $\pi$-lift such that its pushforward measure is equal to $\bigotimes_{b\in B}\nu_b$ --- if $Y$ is constructed using $(X,S)$, $Y'$ using $(X',S)$, and if $X\le X'$, then we have $Y\le Y'$.
\end{rema}

\newcommand{\Sym}{\mathfrak{S}}
Let $\Sym_\pi$ denote the group of all permutations $\sigma$ of $A$ that satisfy $\forall a\in A,\,\pi\circ\sigma(a)=\pi(a)$. This group acts on $[0,\infty]^A$ via $\sigma\cdot x:a\longmapsto x_{\sigma^{-1}(a)}$. We say that a probability measure $\rho$ on $[0,\infty]^A$ is \defini{$\pi$-exchangeable} if, for every $\sigma\in \mathfrak S_\pi$, the pushforward of $\rho$ by $x\mapsto \sigma\cdot x$ is equal to $\rho$.

\begin{enonce}[remark]{Example}
    Sample a random variable $\Theta$. Depending on $\Theta$, for every $b$, choose some probability measure $\nu_{b,\Theta}$ on $[0,\infty]$. Then, conditionally on $\Theta$, let $(Z_a)_{a\in A}$ be a collection of independent random variables such that, for every $a$, the random variable $Z_a$ has distribution $\nu_{\pi(a),\Theta}$. Then, the (unconditional) distribution of $(Z_a)_{a\in A}$ is $\pi$-exchangeable.
\end{enonce}

\begin{enonce}[remark]{Example}
    Let $\pi:A\to B$ be a surjective map between finite sets. Let $k\in \{0,1,\dots,|A|\}$. Pick a random subset of $A$ uniformly at random among all sets of cardinality $k$. Then, the distribution of the indicator function of this random set is $\pi$-exchangeable.
\end{enonce}

More generally, say that $\rho$ is \defini{sufficiently symmetric} if it is possible to choose, for every $b$,  a subgroup $G_b$ of the group of permutations of $\pi^{-1}(b)$ that acts transitively on $\pi^{-1}(b)$ in such a way that for every $\sigma\in \prod_b G_b$, the pushforward of $\rho$ by $x\mapsto \sigma\cdot x$ is equal to $\rho$.

\begin{enonce}[remark]{Example}
    Every $\pi$-exchangeable probability measure $\rho$ is in particular sufficiently symmetric.
\end{enonce}

\begin{enonce}[remark]{Example}
    Let $A=B\times \mathbb{Z}/4\mathbb{Z}$ and let $\pi:(b,x)\longmapsto b$. For every $b$, independently and equiprobably, select either $\{(b,0),(b,2)\}$ or $\{(b,1),(b,3)\}$. Then, the distribution of the indicator function of all selected elements of $A$ is sufficiently symmetric but not $\pi$-exchangeable.
\end{enonce}

Let $\rho$ be a probability measure on $[0,\infty]^A$. For every section $s$, let the \defini{$s$-marginal} of $\rho$ denote the pushforward of $\rho$ by $x\longmapsto (x_{s(b)})_{b\in B}$.
Observe that if $\rho$ is sufficiently symmetric, then all its $s$-marginals are equal. We call this univoquely defined probability measure the \defini{horizontal marginal} of $\rho$.

\begin{coro}
\label{coro:exchange}
Let $\pi :A\to B$ be a surjective map between nonempty countable sets. Let $\mu$ and $\rho$ be two probability measures on $[0,\infty]^A$. Assume that:
\begin{enumerate}
    \item $\mu$ is a $\pi$-lift,
    \item $\rho$ is sufficiently symmetric,
    \item the pushdown of $\mu$ is stochastically dominated by the horizontal marginal of $\rho$.
\end{enumerate}

Then, the probability measure $\mu$ is stochastically dominated by $\rho$.
\end{coro}

In order to prove Theorem~\ref{main}, the key step will be to establish it in the finite setup, namely when $A$ and $B$ are finite and $[0,\infty]$ is replaced by $[N]:=\{0,1,\dots,N\}$. This is performed in Section~\ref{sec:finite}. The general case then follows by general arguments of discrete approximation and compactness --- see Section~\ref{sec:finite-to-gen}.

The finite case is tackled by first taking care of the one-column case, which means that we further assume $B$ to be a singleton. For this, a mere greedy algorithm works --- see Section~\ref{sec:one-column}. To handle the general (still finite) case, we proceed by induction, putting carefully each element of $B$ into play one after the other. Adding one such element is made possible because of Assumption~$\mathbf{A}_\rho$ of Theorem \ref{thm:strongbs} and the one-column case. Performing this strategy requires some care: one important subtlety in the proof is that when a column is yet to be considered, it should be treated as a singleton, not as an empty column.

Before moving on to the proof, we would like to explain why this latter, ``natural'' approach does not suffice. Let us take the one-column case for granted and try to deduce from it the two-columns case, i.e.~the case where $B$ has cardinality 2. For concreteness, let us try to work out the following example. We take $A=\{\mathrm{left},\mathrm{right}\}\times\{\mathrm{top},\mathrm{bottom}\}$, $B=\{\mathrm{left},\mathrm{right}\}$ and $\pi:A\rightarrow B$ the obvious projection.
Let $X$ be a random variable having the uniform distribution on $\{0,1\}^B$, so that its components are independent Bernoulli random variables of parameter $\sfrac12$. Let $\rho$ be the uniform probability measure on $\{0,1\}^A$.
The random section $S$ we consider is the following: if $X=(1,1)$ then $S:i\longmapsto(i,\mathrm{top})$, otherwise we take $S:i\longmapsto(i,\mathrm{bottom})$. Looking only at the first component and using the one-column case, we can stochastically dominate what happens on the left column as desired. To conclude, it would suffice to argue that, conditionally on that, $X_{\mathrm{right}}$ still has a Bernoulli distribution of parameter $\sfrac12$.

At first sight, this may seem to work, as the components of $X$ have been taken to be independent with this distribution. However, in order to take care of the left column, we needed to know not only $X_\mathrm{left}$ but also $S(\mathrm{left})$, which revealed information on $X_{\mathrm{right}}$. In particular, if we lifted a 1 in the top-left position, then we know for sure that $X_{\mathrm{right}}$ is equal to 1. For this reason, adding columns in the naive way is not enough to solve the problem at hand. We have to take into account the fact that the conditional distribution of $X_{\mathrm{right}}$ need not be equal to (or stochastically dominated by) its unconditional distribution. This is why columns that are not handled yet will not be fully forgotten: it will be useful to remember the \emph{values} that are lifted but not \emph{where} they are lifted. 
Assumptions~$\mathbf{A}_\rho$ and $\mathbf{B}_\mu^\rho$ then emerge as the natural setup to suitably perform our ``surgery of couplings''.

At last, let us mention that it is possible to give an explicit example of a suitable coupling on the first column that cannot, in any way, be extended into a suitable coupling on two columns. For instance in our setting, we can consider the following coupling: 
\[
\begin{matrix}
\left[\begin{matrix}1\\ 1\end{matrix}\right] &  \left[\begin{matrix}0\\ 1\end{matrix}\right] & \left[\begin{matrix}1\\ 0\end{matrix}\right] & \left[\begin{matrix}0\\ 0\end{matrix}\right] \\
 & & & \\
\ \ \ \left[\begin{matrix}1 & 1\end{matrix}\right] &  \ \ \ \left[\begin{matrix}1 & 0\end{matrix}\right] & \ \ \ \left[\begin{matrix}0 & 1\end{matrix}\right] & \ \ \ \left[\begin{matrix}0 & 0\end{matrix}\right]
\end{matrix}
\]
This suitable coupling cannot be extended into two columns for the following reason. Conditioned on the 
left column being $\left[\begin{matrix}1\\ 1\end{matrix}\right]$, we know that we need a $1$ on the top right corner for the coupling to work. However, this can only happen with probability at most $\sfrac{1}{2}$, therefore the coupling cannot be extended.

\subsection{The one-column case}
\label{sec:one-column}

Let us start by proving the following easy lemma. For $n\ge0$, we write $[n]:=\{0,1,\dots,n\}$. 

\begin{lemm}
    \label{lem:one-column}
    Let $N$ be a positive integer and let $X$ be an $[N]$-valued random variable. Let $C$ be a finite set and let $\rho$ be a probability measure on $[N]^C$. Assume that for every $c\in C$, the distribution of $X$ is stochastically dominated by the $c$-marginal of $\rho$.

    Let $H$ be a $C$-valued random variable. Let $Y$ denote the random element of $[N]^C$ defined by $Y:c\longmapsto X\,\mathds{1}_{H=c}$. Then, the distribution of $Y$ is stochastically dominated by $\rho$.
\end{lemm}
\newcommand{\Ind}{\mathscr{I}}
\newcommand{\smoller}{\preccurlyeq}
\newcommand{\stoch}{\preccurlyeq}

\begin{proof}
    Without loss of generality, we will assume that $C=[M]$, for some $M$. We consider the set $\Ind=\left(\{1,\dots,N\}\times [M]\right) \sqcup \{0\}$. This set encodes all possible values for $Y$. More precisely, the state $0$ encodes $\varphi_0:=0\in [N]^C$ and $(i,j)$ encodes $\varphi_{i,j}:= i\,\mathds{1}_{\{j\}}$. We endow $\Ind$ with the following well-ordering of lexicographic nature. We declare 0 to be larger than any element of $\Ind$, and $(i,j)\smoller (i',j')$ holds if and only if $i> i'$ or both $i=i'$ and $j\le j'$. Please note that the inequalities for $i$ and $j$ go in reverse direction.
In words, introducing $\smoller$ means that we want to take care of the possible states in the following order:
{\small\[
(N,1)\,\,(N,2)\,\dots\,(N,M)\,\,(N-1,1)\,\,(N-1,2)\,\dots\,(N-1,M)\,\dots\,(1,1)\,\,(1,2)\,\dots\,(1,M)\,\,0.
\]}
The rest of the proof can be summarised as follows: proceeding greedily in this order works.

Let us introduce some formalism. Denote by $\mu$ the distribution of $Y$. A \emph{subcoupling} is a measure $\nu$ on $[N]^C\times[N]^C$ such that the following two conditions hold:
\begin{itemize}
    \item for every $\alpha\in [N]^C$, we have $\sum_\beta \nu(\alpha,\beta)\le \mu(\alpha)$,
    \item for every $\beta\in [N]^C$, we have $\sum_\alpha \nu(\alpha,\beta)\le \rho(\beta)$.
\end{itemize}
Given some $I\in\Ind$, say that a subcoupling is \emph{suitable for step $I$} if it satisfies the following two conditions:
\begin{itemize}
    \item for every $J\smoller I$, we have $\sum_\beta \nu(\varphi_J,\beta)= \mu(\varphi_J)$
    \item for every $\alpha$ which is not of the form $\varphi_J$ for $J\smoller I$, we have $\sum_\beta \nu(\alpha,\beta)=0$.
\end{itemize}
Our purpose is to prove that there is a subcoupling that is suitable for step $0$. If we reach that goal, then we are done. Indeed, as the support of $\mu$ is a subset of $\{\varphi_I\,:\,I\in\Ind\}$, the first marginal will be equal to $\mu$, thus yielding that the total mass of the subcoupling is 1; and as $\rho$ has mass 1 as well, the second marginal will also be equal to $\rho$ instead of satisfying only the inequality of subcouplings.

Let us prove by induction on $(\Ind,\smoller)$ that, for every $I\in\Ind$, there is a subcoupling suitable for step $I$.
We start with $I=(N,1)$. Using the assumption of the lemma for $c=1$, we have $\sum_{\beta:\beta(1)=N}\rho(\beta)\ge \mathbf{P}(X=N)\ge\mathbf{P}(Y=\varphi_{N,1})$, so indeed $\rho$ has enough room in $\{\beta\,:\,\beta(1)=N\}$ to allow such a subcoupling.

Now let $I$ be such that there is a subcoupling suitable for step $I$, and let us pick $\nu$ such a subcoupling. If $I=0$, then we are done, so we assume that $I=(i,j)$. There are three cases to handle:
\begin{enumerate}
    \item when $j<M$,
    \item when $j=M$ and $i>1$,
    \item when $(i,j)=(1,M)$.
\end{enumerate}

Let us consider the first case. We want to prove that there is a subcoupling suitable for step $(i,j+1)$. Using our assumption for $c=j+1$, we have \[\mathrm{Space}:=\sum_{\beta:\beta(j+1)\ge i}\rho(\beta)\ge \mathbf{P}(X\ge i)=\mathbf{P}(X= i)+\sum_{i'>i}\mathbf{P}(X= i'),\] where the last sum is the null empty sum if $i=N$. Let us compute how much of this space has been used already, and how much we need to use now. We need to store a mass of $\mathrm{Need}:=\mathbf{P}(X=i,\,H=j+1)$. On the other hand, the total mass of $\nu$ is $\mathrm{Used}:=\sum_{j'\le j}\mathbf{P}(X= i,\,H=j')+\sum_{i'>i}\mathbf{P}(X= i')$. Therefore, the inequality \[\mathbf{P}(X=i)\ge \mathbf{P}(X=i,\,H=j+1)+ \sum_{j'\le j}\mathbf{P}(X= i,\,H=j')\] guarantees that $\mathrm{Space}-\mathrm{Used}\ge\mathrm{Need}$, thus yielding the existence of a subcoupling that is suitable for step $(i,j+1)$.

Let us move on to the second case. We want to prove that there is a subcoupling suitable for step $(i-1,1)$. Let us reset the notation $(\mathrm{Space},\mathrm{Need},\mathrm{Used})$. By assumption, we have $\mathrm{Space}:=\sum_{\beta:\beta(1)\ge i-1}\rho(\beta)\ge \mathbf{P}(X\ge i-1)=\mathbf{P}(X= i-1)+\sum_{i'\ge i}\mathbf{P}(X= i')$. We need to store a mass of $\mathrm{Need}:=\mathbf{P}(X=i-1,\,H=1)$, and the total mass of $\nu$ is $\mathrm{Used}:=\sum_{i'\ge i}\mathbf{P}(X= i')$. Therefore, the trivial inequality \[\mathbf{P}(X=i-1)\ge \mathbf{P}(X=i-1,\,H=1)\] guarantees that $\mathrm{Space}-\mathrm{Used}\ge\mathrm{Need}$, thus yielding the existence of a subcoupling that is suitable for step $(i-1,1)$.

At last, we need to take care of the case $(i,j)=(1,M)$. Then, the only remaining configuration to handle is the null configuration. Since it is the minimal configuration, it imposes no constraint at all on the mass repartition, so this step holds trivially. We could also say that $\mathrm{Space}:=1$, $\mathrm{Need}:=\mathbf{P}(X=0)$, $\mathrm{Used}:=\sum_{i\ge 1}\mathbf{P}(X=i)$, and that the inequality  $\mathrm{Space}-\mathrm{Used}\ge\mathrm{Need}$ clearly holds --- and is an equality. This concludes the proof.
\end{proof}

\begin{rema}
    Lemma~\ref{lem:one-column} can also be proved using Strassen's Theorem for monotone couplings \cite{s,l,a}, which itself admits simple proofs for finite spaces --- for instance using Farkas' Lemma or Hall's Marriage Theorem. We deemed more instructive to provide the above proof, which requires no tool at all and provides a straightforward construction of a suitable coupling.
\end{rema}

\subsection{From one column to the finite case}
\label{sec:finite}

The goal of this section is to prove Proposition~\ref{finite}, which is a finite version of Theorem~\ref{main}. The reason why we restrict ourselves to a finite setup is to avoid technical questions of measurability.

\begin{prop}\label{finite}
Let $N$ be a positive integer. Let $\pi: A\rightarrow B$ be a surjective map between nonempty finite sets. Let $\mathfrak s$ be a  section of $\pi$. Let $\mu$ and $\rho$ be two probability measures on $[N]^A$. Assume that $\mu$ is a $\pi$-lift and that the conditions $\mathbf{A}_\rho$ and $\mathbf{B}_\mu^\rho$ are met.
Then, the probability measure $\mu$ is stochastically dominated by $\rho$.
\end{prop}

We interpret $A$ as an array with columns indexed by $B$, and $\mathfrak s$ as selecting one favourite position in each column. The condition $\mathbf{B}_\mu^\rho$ ensures that the ``$\mathfrak s$-flattened'' version of $\mu$ is dominated by $\rho$. This fact will serve as the base case of an induction proof to show that we can ``unflatten'' the columns one by one while keeping stochastic domination by $\rho$. More precisely, we will use the stochastic domination of the ``flattened'' measure to dominate the one before flattening. To do so, we need two tools. The first one is a way to pull back our monotone coupling, which is the object of Lemma~\ref{extension}. The second one is a way to use Lemma~\ref{lem:one-column} in order to resample a column in a monotonic way, without breaking the rest of our coupling. This is enabled by Lemma~\ref{integration de couplage monotone}.

It will be convenient to study couplings for probability measures. Given $\xi_1$ and $\xi_2$ two probability measures on spaces $E_1$ and $E_2$, a \defini{coupling} of $\xi_1$ and $\xi_2$ is a probability measure $\eta$ on $E_1\times E_2$ such that, for every $i$, the pushforward of $\eta$ by $(x_1,x_2)\longmapsto x_i$ is $\xi_i$.

\begin{lemm}\label{extension}
    Let $h_1:E_1\to E'_1$ and $h_2:E_2\to E'_2$ be two functions between countable sets. Let $\nu_1$ be a probability measure on $E_1$ and $\nu_2$ be a probability measure on $E_2$. Let $\eta$ be a coupling between the pushforward measures $\nu_1\circ h_1^{-1}$ and $\nu_2\circ h_2^{-1}$.
    
    Then, there is a coupling between $\nu_1$ and $\nu_2$ such that its pushforward by $(h_1,h_2)$ is $\eta$.
\end{lemm}

\begin{proof}
Let $X_1$ be a $\nu_1$-distributed random variable and $X_2$ be a $\nu_2$-distributed random variable. For every $(x'_1,x'_2)\in E'_1\times E'_2$ such that $\eta(x'_1,x'_2)>0$, we consider some probability measure on $E_1\times E_2$ with first marginal the conditional distribution of $X_1$ given $h_1(X_1)=x'_1$ and  with second marginal the conditional distribution of $X_2$ given $h_2(X_2)=x'_2$ --- such a probability measure exists, one can for example take the product of these measures. The lemma follows by considering the convex combination of these measures, with convexity coefficients given by $\eta(x'_1,x'_2)$.
\end{proof}

Lemma~\ref{extension} can be used to reprove transitivity of stochastic domination. In this section, we use the notation $\stoch$ for stochastic domination: we write $\xi\stoch\zeta$ to mean that $\xi$ is stochastically dominated by $\zeta$.

\begin{coro}
Let $\xi_1$, $\xi_2$ and $\xi_3$ be three probability measures on a countable ordered space $E$ such that $\xi_1 \stoch \xi_2$ and $\xi_2 \stoch \xi_3$. Then, we have $\xi_1 \stoch \xi_3$.
\end{coro}

\begin{proof}
Take $\nu_1$ to be the distribution of a monotone coupling between $\xi_1$ and $\xi_2$, and set $\nu_2 := \xi_3$. Let $h_1: E\times E \rightarrow E$ be the projection on the second coordinate. As we assumed $\xi_2 \stoch \xi_3$, we can pick $\eta$ the distribution of some monotone coupling  between $\nu_1 \circ h_1^{-1} = \xi_2$ and $\nu_2$.
Applying Lemma~\ref{extension} with $h_2=\mathrm{id}_E$, we get a triplet of random variables $(X_1,X_2,X_3)$ such that $(X_1,X_2)$ is $\nu_1$-distributed and $(X_2,X_3)$ is $\eta$-distributed. As both $\nu_1$ and $\eta$ are monotone couplings, the chain of inequalities $X_1\leq X_2$ and $X_2 \leq X_3$ holds almost surely. Therefore, $(X_1,X_3)$ is a monotone coupling between $\xi_1$ and $\xi_3$.
\end{proof}

The next lemma enables us to integrate monotonone couplings.

\begin{lemm}\label{integration de couplage monotone}
\label{lem:conditional}
Let $C$ be a countable ordered set and $D$ a countable set. Let $f$ be a map from $C\times C$ to $D$. Let $\nu_1$ and $\nu_2$ be probability measures on $C$. Let $(X_1, X_2)$ be a coupling of $\nu_1$ and $\nu_2$. This coupling is not assumed to be monotone.\\
Assume that for all $d\in D$ such that $E_d := \{f(X_1, X_2) = d\}$ has positive probability, there is a monotone (increasing) coupling $(X^d_1, X^d_2)$ between the conditional distribution of $X_1$ given $E_d$ and that of $X_2$ given $E_d$.\\
Then, the probability measure $\nu_1$ is stochastically dominated by $\nu_2$.
\end{lemm}
\begin{proof}
For every $d$ such that $\mathbb{P}(E_d)>0$, let $\rho_d$ be the distribution of a suitable $(X^d_1,X^d_2)$. The distribution $\sum_d \mathbb{P} (E_d) \rho_d$ is then that of a monotone coupling between $\nu_1$ and $\nu_2$.
\end{proof}

We are now ready to establish Proposition~\ref{finite}.

\begin{proof}[Proof of Proposition~\ref{finite}]
Take $(A,B,\pi,\mathfrak{s},\mu,\rho)$ as in Proposition~\ref{finite}. Without loss of generality, we assume that $B = \llbracket 1, K\rrbracket:=\{1,\dots,K\}$. Our goal is to prove that $\rho \succcurlyeq \mu$.
To do so, we will work by considering ``flattened'' versions of $\mu$, and recursively ``unflatten'' them in a proper coupling. 


We define a sequence of measures $(\mu_n)_{n\in \llbracket 0, K\rrbracket}$ by induction from $K$ until 0. We set $\mu_K:=\mu$. For $n\in \llbracket 1,K\rrbracket$, let $f_n : \llbracket 0, N\rrbracket^A \rightarrow \llbracket 0, N\rrbracket^A$ be defined as follows
\begin{itemize}
    \item if $\pi(a)\neq n$, then $f_n(y)_a=y_a$,
    \item if $a=\mathfrak{s}(n)$, then $f_n(y)_a=\max_{a'\in \pi^{-1}(n)} y_{a'}$,
    \item otherwise, $f_n(y)_a=0$.
\end{itemize}
For $n\in\llbracket 0,K-1\rrbracket$, the measure $\mu_n$ is then defined as the pushforward of $\mu_{n+1}$ by $f_{n+1}$.
By definition and because of Assumption~$\mathbf{B}_\mu^\rho$, we have $\mu_0 \stoch \rho$. Let us now establish the induction step $(\mu_n \stoch \rho)\implies (\mu_{n+1} \stoch \rho)$.

Let $n\in\llbracket0, K-1\rrbracket$ be such that $\mu_n \stoch \rho$. Our goal is to show that $\mu_{n+1} \stoch \rho$.
Let us apply Lemma~\ref{extension} with $E_1=E_1'=E_2=E'_2=\llbracket 0, N\rrbracket^A$, $h_1=f_{n+1}$, $h_2$ the identity map, $\nu_1=\mu_{n+1}$, $\nu_2=\rho$, and $\eta$ a monotone coupling between $\mu_n$ and $\rho$ --- such an $\eta$ exists as we assumed $\mu_n\stoch \rho$.
This enables us to take a coupling $(Y,Z)$ of $\mu_{n+1}$ and $\rho$ such that the following is true. If $\ell_{n+1}$ is set to be $\text{argmax}_{a\in \pi^{-1}(n+1)}(Y_a)$ when it is well defined and $\mathfrak s(n+1)$ otherwise then, we almost surely have:
\begin{itemize}    
    \item $Y_a = 0$ for every $a\in \pi^{-1}(n+1)\backslash \{\ell_{n+1}\}$,
    \item $Z_a\geq Y_a$ for every $a\neq \ell_{n+1}$,
    \item $Z_{\mathfrak s(n+1)}\geq Y_{\ell_{n+1}}=\max_{a\in\pi^{-1}(n+1)}Y_a$.
\end{itemize}
Note that this coupling has no reason to be monotone, so more work is needed to deduce stochastic domination.

Let us consider the forgetful function $g_{n+1}:\llbracket 0, N\rrbracket^{A}\times \llbracket 0, N\rrbracket^{A}\longrightarrow \llbracket 0, N\rrbracket^{A\backslash\pi^{-1}(n+1)}$ defined by $((y_a)_{a\in A},(z_a)_{a\in A})  \longmapsto (z_a)_{a\notin \pi^{-1}(n+1)}$.
We want to apply Lemma~\ref{lem:conditional} with $g_{n+1}$, $Z$ and $Y$. Before doing so, let us gather a few observations:
\begin{enumerate}
    \item Since we almost surely have domination outside $\pi^{-1}(n+1)$, for any given $d\in \llbracket 0, N\rrbracket^{A\backslash\pi^{-1}(n+1)}$, the random variable $Y$ conditioned on the event $E_d:= \{g_{n+1} (Y, Z) = d\}$ is almost surely dominated by $d$ outside $\pi^{-1}(n+1)$.
    \item\label{item:couple} Thanks to this and Lemma~\ref{extension}, for any given $d$, we can find a monotone coupling $(Y^d, Z^d)$ as soon as there is stochastic domination of the conditional distribution of $(Y_a)_{a\in \pi^{-1}(n+1)}$ given $E_d$ by that of $(Z_a)_{a\in \pi^{-1}(n+1)}$ given $E_d$.
    \item By construction, the distribution of $(Y_a)_{a\in \pi^{-1}(n+1)}$ conditioned on $E_d$ is a $\pi$-lift.
    \item\label{item:subtle} As the inequality  $\max_{a\in \pi^{-1}(n+1)} Y_a\le Z_{\mathfrak s(n+1)}$ holds almost surely, it also holds almost surely conditioned on $E_d$. In particular, the distribution of $\max_{a\in \pi^{-1}(n+1)} Y_a$ conditioned on $E_d$ is stochastically dominated by that of $Z_{\mathfrak s(n+1)}$ conditioned on $E_d$.
    \item\label{item:aass} Finally, Assumption $\mathbf{A}_\rho$ assures us that for any $a\in \pi^{-1}(n+1)$, conditioned on $E_d$, the distribution of $Z_a$ stochastically dominates that of $Z_{\mathfrak s(n+1)}$.
\end{enumerate}

Due to Observations~\ref{item:subtle} and \ref{item:aass}, transitivity of stochastic domination ensures that for any given $d$ and every $a\in \pi^{-1}(n+1)$, conditioned on $E_d$, the distribution of $Z_a$ stochastically dominates that of $\max_{\pi^{-1}(n+1)} Y$. This allows us to use Lemma~\ref{lem:one-column}, which yields that the distribution of $(Y_a)_{a\in\pi^{-1}(n+1)}$ conditioned on $E_d$ is stochastically dominated by that of $(Z_a)_{a\in\pi^{-1}(n+1)}$ conditioned on $E_d$. Coupled with Observation~\ref{item:couple}, this result implies that conditioned on any event of the form $E_d$, the distribution of $Z$ stochastically dominates that of $Y$. We can now use Lemma~\ref{integration de couplage monotone} to conclude that the distribution of $Z$ --- in other words $\rho$ --- stochastically dominates that of $Y$ --- that is, $\mu_{n+1}$.
The induction step is thus established. By induction, we get $\rho\succcurlyeq\mu_K=\mu$.
\end{proof}

\subsection{From the finite setup to the general case}
\label{sec:finite-to-gen}

Let us now use Proposition~\ref{finite} to establish Theorem~\ref{main}. Naturally, the strategy is to discretise the problem, apply Proposition~\ref{finite}, and then take limits.

For $n\ge 0$, let $f_n$ be the function defined on $[0,\infty]$ by
\[f_n:x\longmapsto\max\left\{\tfrac k n:\,\tfrac k n\le x\text{ and }0\le k\le n^2\right\}.\] Each $f_n$ is weakly increasing and the sequence $(f_n)$ converges pointwise to the identity of $[0,\infty]$.

Now, let us enumerate $A=\{a_1,a_2,\dots\}$ in such a way that for all $i$ and $j$, we have $a_i=\mathfrak{s}\circ \pi(a_j)\implies i\le j$. Such an enumeration can for instance be obtained by starting with an arbitrary enumeration and then, for each $b\in B$, swapping $\mathfrak{s}(b)$ with the first element of $\pi^{-1}(b)$ to appear in this enumeration. For $n\ge0$, define the map $g_n:[0,\infty]^A\to [0,\infty]^A$ by $g_n(x):a_i\mapsto f_n(x_{a_i})\,\mathds{1}_{i\le n}$. Each map $g_n$ is weakly increasing and the sequence $(g_n)$ converges pointwise to the identity of $[0,\infty]^A$. For $n\ge 0$, define the probability measure $\mu_n$ (resp. $\rho_n$) to be the pushforward of $\mu$ (resp. $\rho$) by $g_n$. In other words, if a random variable $Y$ has distribution $\mu$ (resp. $\rho$), then the distribution of $g_n(Y)$ is denoted by $\mu_n$ (resp. $\rho_n$).

Observe that for every $n$, the measures $\mu_n$ and $\rho_n$ satisfy the assumptions of Theorem~\ref{main}, as we assumed that it was the case for $\mu$ and $\rho$. On top of that, $(\mu_n,\rho_n)$ fits the finite setup: the entries outside of $\{a_1,\dots,a_n\}$ are all null and irrelevant; also notice that if Proposition~\ref{finite} holds for all $\{0,1,\dots,N\}$, then it holds for all $\{0,1,\dots,N^2\}$, hence also for labels in $\{\frac k N,:\, 0\le k\le N^2\}$. Therefore, we can apply Proposition~\ref{finite} and get some monotone coupling $\nu_n$ between $\mu_n$ and $\rho_n$.

The space $[0,\infty]^A$, endowed with the product topology, is compact Hausdorff, thus so is its Cartesian square. The set $\{(x,y)\in [0,\infty]^A\times [0,\infty]^A\,:\,\forall a\in A,\, x_a\le y_a\}$ is closed inside this square, hence compact itself. Therefore, the space of monotone couplings, endowed with the smallest topology making continuous the integration against all continuous bounded functions, is compact. We can thus pick $\nu$ some accumulation point of the sequence $(\nu_n)$: we may write $\nu=\lim_{k\to\infty}\nu_{n_k}$. By construction, $\nu$-almost every $(x,y)$ satisfies $\forall a\in A,\,x_a\le y_a$. What remains to check is that $\nu$ has the correct marginals.

By symmetry, it suffices to take care of the first marginal. It suffices to check that for every continuous (automatically bounded) $h:[0,\infty]^A\to \mathbb{R}$, we have $\int h(x)\,\mathrm{d}\nu(x,y)=\int h(x)\,\mathrm{d}\mu(x)$. As $h$ is continuous bounded, we know that \[\int h(x)\,\mathrm{d}\nu(x,y)=\lim_{k\to\infty}\int h(x)\,\mathrm{d}\nu_{n_k}(x,y).\] As the first marginal of $\nu_{n_k}$ is $\mu_{n_k}$, we have $\int h(x)\,\mathrm{d}\nu(x,y)=\lim_{k\to\infty}\int h(x)\,\mathrm{d}\mu_{n_k}(x)$. It thus suffices to establish that $(\mu_n)$ converges to $\mu$. To see this, let $Y$ be a $\mu$-distributed random variable: as $g_n(Y)$ converges pointwise to $Y$, the distribution of $g_n(Y)$ converges to that of $Y$, which is exactly what we had left to prove. This concludes the proof of Theorem~\ref{main}.\qed

\subsection{Proofs of the corollaries}
\label{sec:coro}
This section is dedicated to deriving Corollaries~\ref{coro:indep} and \ref{coro:exchange} from Theorem~\ref{main}.

\subsubsection{Proof of Corollary~\ref{coro:indep}} Without loss of generality, we assume that $A$ and $B$ are disjoint. Let $\tilde A:=A\sqcup B$. Let $\tilde\pi:\tilde A \to B$ be defined as $\pi$ on $A$ and as the identity on $B$. For every $b\in B$, pick some probability measure $\tilde\rho_b$ on $\tilde\pi^{-1}(b)=\pi^{-1}(b)\sqcup\{b\}$ in such a way that the $\pi^{-1}(b)$-marginal of $\tilde\rho_b$ is $\rho_b$ and its $b$-marginal is $\nu_b$. This is indeed possible, for instance by setting $\tilde\rho_b:=\rho_b\otimes \nu_b$. We consider the section $\mathfrak s$ of $\tilde \pi$ given by the inclusion of $B$ as a subset of $\tilde A=A\sqcup B$. Let $\tilde\rho:=\bigotimes_b\tilde\rho_b$ and let $\tilde\mu$ be the pushforward of $\mu$ by the map $g:[0,\infty]^A\to[0,\infty]^{A\sqcup B}$ given by
\[
    g(x)_c=\begin{cases}
        x_c&\text{if $c\in A$,}\\
        0&\text{if $c\in B$.}
    \end{cases}
\]
We want to apply Theorem~\ref{main} to $(\tilde A,B,\tilde \pi,\mathfrak s,\tilde\rho,\tilde\mu)$.

It is the case that $\tilde\pi$ is a surjective map between nonempty countable sets. The probability measure $\tilde\mu$ is indeed a $\tilde\pi$-lift, as the map $g$ attributes automatically the value zero to all entries that were not originally present in $A$.
Assumption~$\mathbf{A}_{\tilde\rho}$ holds: indeed, by independence, there is no conditioning to make, and we assumed that for every $b$ and every $a\in\pi^{-1}(b)$, the probability measure $\nu_b$ is stochastically dominated by the $a$-marginal of $\rho_b$.
Besides, the last assumption of Corollary~\ref{coro:indep} guarantees that the pushdown of $\tilde\mu$ is stochastically dominated by $\bigotimes_b\nu_b$, which gives Assumption~$\mathbf{B}_{\tilde\mu}^{\tilde\rho}$.

We can thus apply Theorem~\ref{main}, yielding that $\tilde\mu$ is stochastically dominated by $\tilde\rho$. Restricting a corresponding coupling to entries in $A$ instead of the whole $A\sqcup B$ establishes the stochastic domination of $\mu$ by $\rho$, hereby completing the proof of Corollary~\ref{coro:indep}. \qed

\subsubsection{Proof of Corollary~\ref{coro:exchange}} Let $\mathfrak s$ be any section of $\pi$. It suffices to check Assumptions~$\mathbf A_\rho$ and $\mathbf B^\rho_\mu$.
Observe that for every $b$, the fixator\footnote{When a group $G$ acts on a set $E$, the \defini{fixator} of a subset $P$ of $E$ is the group of all $g$'s satisfying $\forall x\in P,\,g\cdot x=x$. Here, we consider $\prod_b G_b$ acting on $A$.} of $A\setminus\pi^{-1}(b)$ can be identified with $G_b$ and acts transitively on $\pi^{-1}(b)$. Therefore, the conditional distribution appearing in $\mathbf A_\rho$ depends only on $b$, not on $a\in\pi^{-1}(b)$. Assumption~$\mathbf A_\rho$ follows, as a probability measure is always stochastically dominated by itself. As for $\mathbf{B}_\mu^\rho$, it follows from our third assumption. \qed

\begin{rema}
    In the previous argument, the condition that really matters as far as $\mathbf A_\rho$ is concerned is the following transitive-fixator condition. Consider the group $\mathrm{Sym}_\rho$ of all $\sigma\in \mathfrak{S}_\pi$ such that the pushforward of $\rho$ by $x\mapsto \sigma\cdot x$ is equal to $\rho$. The condition then reads as follows: $\mathrm{Sym}_\rho$ satisfies that, for every $b$, the fixator of $A\setminus\pi^{-1}(b)$ acts transitively on $\pi^{-1}(b)$. This condition is indeed satisfied as soon as $\rho$ is sufficiently symmetric. Let us explain why, conversely, the transitive-fixator condition implies that $\rho$ is sufficiently symmetric. For each $b$, using the transitive-fixator condition yields a transitive subgroup $G_b$ of the permutations of $\pi^{-1}(b)$ such that letting it act naturally on $\pi^{-1}(b)$ and as the identity elsewhere preserves the probability measure $\rho$. We deduce that $\rho$ is left invariant under the action of the group $\bigoplus_b G_b$ --- which is defined to be the group of the elements $(g_b)$ of $\prod_b G_b$ such that all but finitely many $g_b$'s are the identity. Finally, as probability measures on $[0,\infty]^A$ are characterised by their finite-dimensional marginals, invariance under $\bigoplus_b G_b$ implies invariance under the larger group $\prod_b G_b$. In the end, the following three conditions are equivalent: the transitive-fixator condition, the existence of transitive $G_b$'s such that $\mathrm{Sym}_\rho$ contains $\bigoplus_b G_b$, and the same with $\prod_b G_b$ instead of $\bigoplus_b G_b$. We have chosen to define ``sufficiently symmetric'' by using $\prod_b G_b$ because this definition is simple and because it makes clear that the horizontal marginal is well defined.
\end{rema}

\section{Variations on Theorem~\ref{main}}
\label{sec:gen}

In this section, we investigate in which directions Theorem~\ref{main} may or may not be improved. This section is not used in subsequent sections.

\subsection{Lifting several variables per column}
\label{sec:multi}
In this subsection, we do not try to give a maximally general statement fitting the purpose announced in the title --- this would probably be hard to read. We will see that a new phenomenon occurs, and we have chosen to focus on the simplest statement exhibiting this phenomenon.

The purpose will no longer be to prove that two fully specified probability measures satisfy stochastic domination. Instead, the lifting process will leave some room for indeterminacy, and the theorem will state that there is a way to proceed such that the lifting constraints are satisfied and stochastic domination holds.

A convenient way to state our result uses the notion of extension of probability spaces. An \defini{extension} of a probability space $(\Omega,\mathscr{F},\mathbb{P})$ is the data of a probability space $(\Omega^\star,\mathscr{F}^\star,\mathbb{P}^\star)$ together with a measure-preserving map $\varphi:\Omega^\star\to \Omega$, i.e.~a measurable map such that the pushforward of $\mathbb{P}^\star$ by $\varphi$ is $\mathbb P$. Any random variable $X$ on $(\Omega,\mathscr{F},\mathbb{P})$ can be interpreted in the extended space by considering $X^\star:=X\circ \varphi$. By abuse of notation, if the context is unambiguous, we may continue to use the notation $X$ instead of $X^\star$ for extended random variables. 

\begin{theo}
    \label{thm:multilift}
    Let $\pi:A\to B$ be a surjective map between nonempty countable sets. Let $p\in [0,1]$. For every $b\in B$, let $X_b$ and $X^\dagger_b$ be Bernoulli random variables of parameter $p$. Assume that the collection of random variables formed by all the $X_b$'s and all the $X^\dagger_b$'s is independent.
    Let $S$ and $S^\dagger$ be random sections of $\pi$ such that, for every $b\in B$, we almost surely have $S(b)\neq S^\dagger(b)$.

    Then, up to extending the probability space, there are $\{0,1\}$-valued random variables $Y_a$ and $Z_a$ such that the following conditions hold:
    \begin{enumerate}
        \item \begin{enumerate}
             \item\label{item:useless1} for every $a$, we have $a\notin \{S\circ\pi(a),S^\dagger\circ\pi(a)\}\implies Y_a=0$,
            \item\label{item:useless2} for every $b$ satisfying $X_b=0$ and $X^\dagger_b=0$, we have $Y_{S(b)}=0$ and $Y_{S^\dagger(b)}=0$,
            \item \label{item:strong}for every $b$ satisfying $X_b=1$ and $X^\dagger_b=0$, we have $Y_{S(b)}=1$ and $Y_{S^\dagger(b)}=0$,
            \item for every $b$ satisfying $X_b=1$ and $X^\dagger_b=1$, we have $Y_{S(b)}=1$ and $Y_{S^\dagger(b)}=1$,
            \item \label{item:indeterminacy}for every $b$ satisfying $X_b=0$ and $X^\dagger_b=1$, exactly one $a$ in $\left\{S(b),S^\dagger(b)\right\}$ satisfies $Y_a=1$,
        \end{enumerate}
        \item the distribution of $(Z_a)$ is $\mathrm{Bernoulli}(p)^{\otimes A}$,
        \item for all $a$, we have $Y_a\le Z_a$.
    \end{enumerate}
\end{theo}

\begin{rema}
    The indeterminacy is captured by Item~\ref{item:indeterminacy}, where we have the possibility to lift the ``weak bit'' $X^\dagger_{b}$ in $\{S(b),S^\dagger(b)\}$ but not to decide where. On the contrary, Item~\ref{item:strong} states that the ``strong bit'' $X_{b}$ can be lifted exactly where one wants. Items~\ref{item:useless1} and \ref{item:useless2} are useless, in the sense that if some $(Y_a)$ satisfies the conclusion apart from these items, setting ${Y}'_a=Y_a\, \mathds{1}_{a\in \{S\circ\pi(a),S^\dagger\circ\pi(a)\}}\,\mathds{1}_{ (X_{\pi(a)},X^\dagger_{\pi(a)})\neq (0,0)}$ will satisfy the full conclusion for free. Likewise, in Item~\ref{item:indeterminacy}, it is not ``exactly one'' that truly matters but ``at least one''. We decided to include these conditions in order to try and construct $Y_a$ as explicitly as before, and see exactly where it fails. If we remove all useless items from the conclusion, then we can further take $Y_a=Z_a$.
\end{rema}

Our motivation for investigating this was to get a new proof of Theorem~\ref{thm:strict} using techniques in the spirit of Theorem~\ref{lakon}. Theorem~\ref{thm:strict}, first obtained in \cite{ms}, is a version of Theorem~\ref{thm:bs} where, under additional assumptions, one gets the strict inequality. The general idea there is to use the fact that $\pi$ is (at least 2)-to-1 to try and ``get a second chance''. We have not managed to use Theorem~\ref{thm:multilift} or variations of it to obtain such a proof. Still, Section~\ref{sec:strict} provides a new proof of this result, based on explorations and couplings.

The strategy to prove Theorem~\ref{thm:multilift} follows an easy scheme: from some $\pi:A\to B$ and probability measures $\mu$ and $\rho$, construct new $\tilde \pi:\tilde A\to B$ and probability measures $\tilde \mu$ and $\tilde \rho$; then apply Theorem~\ref{main} to these new objects. It will be easy for the reader to craft their own variations.

\begin{proof}[Proof of Theorem~\ref{thm:multilift}]
    Let $\tilde A:=\{(a_1,a_2)\in A^2\,:\,a_1\neq a_2\text{ and }\pi(a_1)=\pi(a_2)\}$. Let $\tilde \pi:\tilde A\to B$ be defined by mapping any $(a_1,a_2)$ to the element $\pi(a_1)$, which is also $\pi(a_2)$. The set $\tilde A$ is indeed countable. As $\tilde S:b\longmapsto (S(b),S^\dagger(b))$ is a random section of $\tilde \pi$, such sections exist, entailing surjectivity of $\tilde \pi$ --- and nonemptiness of $\tilde A$, as $B$ is nonempty.

For every $b\in B$, let $\tilde X_b:=2X_b+X^\dagger_b$. For every $b\in B$, let $\rho_b:=\mathrm{Bernoulli}(p)^{\otimes\pi^{-1}(b)}$, and let $\tilde \rho_b$ be the pushforward of the measure $\rho_b$ by the map \[(x_a)_{a\in\pi^{-1}(b)}\longmapsto (2x_{a_1}+x_{a_2})_{(a_1,a_2)\in \tilde \pi^{-1}(b)}.\] As the $\tilde X_b$'s are independent, we can apply Corollary~\ref{coro:indep-va} to $\tilde X_b$, $\tilde S(b)$, and $\tilde \rho_b$. Therefore, letting $\tilde Y_\mathbf{a}:=\tilde X_{\tilde \pi(\mathbf{a})}\,\mathds{1}_{\tilde S\circ\tilde \pi(\mathbf{a})=\mathbf{a}}$, we get that the distribution of $(\tilde Y_\mathbf{a})$ is stochastically dominated by $\bigotimes_{b\in B}\tilde \rho_b$.

Up to extending our probability space, we can thus define random variables $Z_a$, for $a\in A$, in such a way that:
\begin{itemize}
    \item the random variables $Z_a$ are i.i.d. with Bernoulli distribution of parameter $p$,
    \item for every $\mathbf{a}=(a_1,a_2)\in \tilde A$, we have $\tilde S\circ {\tilde \pi(\mathbf{a})}=\mathbf{a}\implies \tilde Y_a\le 2Z_{a_1}+Z_{a_2}$.
\end{itemize}
We defer justification of this point to the end of the proof.
For now, observe that the second condition can be rewritten as
\[\tilde S\circ {\tilde \pi(\mathbf{a})}=\mathbf{a}\implies (2X_{\tilde \pi(\mathbf{a})}+X^\dagger_{\tilde \pi(\mathbf{a})})\,\mathds{1}_{\tilde S\circ {\tilde \pi(\mathbf{a})}=\mathbf{a}}\le 2Z_{a_1}+Z_{a_2}\]
holding for every $\mathbf{a}$, which in turn can be written as
\[
2Z_{S(b)}+Z_{S^\dagger(b)}\ge 2X_{b}+X^\dagger_{b}
\]
holding for every $b$.

Observe that, to establish the result, there is no need to speak of the $Y_a$'s. It suffices to prove Items~\ref{item:useless1}--\ref{item:indeterminacy} with lower bounds on $Z_a$ instead of equalities for $Y_a$. More precisely, it suffices to prove the following three conditions:
\begin{itemize}
    \item for every $b$ satisfying $X_b=1$ and $X^\dagger_b=0$, we have $Z_{S(b)}=1$,
    \item for every $b$ satisfying $X_b=1$ and $X^\dagger_b=1$, we have $Z_{S(b)}=1$ and $Z_{S^\dagger(b)}=1$,
    \item for every $b$ satisfying $X_b=0$ and $X^\dagger_b=1$, at least one $a$ in $\left\{S(b),S^\dagger(b)\right\}$ satisfies $Z_a=1$,
\end{itemize}
Let us establish that these conditions holds. When $(X_{b},X^\dagger_{b})=(1,1)$, we have $2Z_{S(b)}+Z_{S^\dagger(b)}\ge 3$, which gives $Z_{S(b)}=Z_{S^\dagger(b)}=1$, as desired. When $(X_{b},X^\dagger_{b})=(1,0)$, we have $2Z_{S(b)}+Z_{S^\dagger(b)}\ge 2$, which indeed guarantees that $Z_{S(b)}=1$. At last, when $(X_b,X^\dagger_b)=(0,1)$, we have $2Z_{S(b)}+Z_{S^\dagger(b)}\ge 1$, which indeed guarantees that there is at least one 1 among $Z_{S(b)}$ and $Z_{S^\dagger(b)}$. 

It remains to justify how to implement stochastic domination via extensions, which is in the spirit of Lemma~\ref{extension}. Let $\mathbf{Z}=(Z_a)_{a\in A}$ be a family of independent Bernoulli random variables of parameter $p$, defined on some probability space. On this same probability space, we can also define $\mathbf{W}=(Z_{a_1},Z_{a_2})$, which has distribution $\bigotimes_{b\in B}\tilde \rho_b$. As the random variables $\mathbf{Z}$ and $\mathbf{W}$ take values in Polish spaces, we can define a probability kernel $\kappa$ such that the conditional distribution of $\mathbf{Z}$ given $\mathbf{W}$ is $\kappa_\mathbf{W}$. 
On our original probability space, we have the random variable $\mathbf{Y}=(\tilde Y_\mathbf{a})$. By Corollary~\ref{coro:indep-va}, we know that there are yet another probability space and random variables $\mathbf{Y}'$ and $\mathbf{W}'$ defined on it such that $\mathbf{Y}'$ has the same distribution as $\mathbf{Y}$, $\mathbf{W}'$ the same distribution as $\mathbf{W}$, and the inequality $\mathbf{Y}'\le \mathbf{W}'$ holds almost surely. Let $\kappa'$ be a probability kernel such that the conditional distribution of $\mathbf{W}'$ given $\mathbf{Y}'$ is $\kappa'_{\mathbf{Y}'}$.
If $\mathbf{Z}$ takes values in $E$, then a suitable extension is given by $\Omega^\star:=\Omega\times E$, the probability measure $\mathbb{P}^\star:=\iint \delta_{\omega}\otimes \kappa_{\mathbf{w}}\,\mathrm{d}\kappa'_{\mathbf{Y}(\omega)}(\mathbf{w})\,\mathrm{d}\mathbb{P}(\omega)$, and the map $\varphi:(\omega,\mathbf{z})\longmapsto \omega$. The random variable $\mathbf{Z}^\star:(\omega,\mathbf{z})\longmapsto \mathbf{z}$ behaves as desired.
\end{proof}

\begin{enonce}[remark]{Counterexample}
    Please note that Theorem~\ref{thm:multilift} does not hold if the condition of the subtle Item~\ref{item:indeterminacy} is replaced by ``$Y_{S(b)}=0$ and $Y_{S^\dagger(b)}= 1$''. Indeed, a counterexample is possible with $B=\{o\}$ and $A=\{1,2\}$. Simply take $(S(o),S^\dagger(o))=(1,2)$ if $X_{o}=1$ and $(S(o),S^\dagger(o))=(2,1)$ otherwise. Then, the parameter of the Bernoulli random variable lifted at 1 via this too strong procedure would have parameter $1-(1-p)^2$, which is strictly larger than $p$ as soon as $p\in(0,1)$.
\end{enonce}

\subsection{Beyond Conditions~$\mathbf{A}_\rho~\mathbf{B}_\mu^\rho$}
\label{sec:somesections}

In this section, we investigate a generalisation of Theorem~\ref{main} that may seem reasonable but turns out to be wrong. We show that if we drop Assumptions~$\mathbf{A}_\rho~\mathbf{B}_\mu^\rho$ and assume something weaker instead, then Theorem~\ref{main} becomes false. The weak assumption we consider is:
\begin{itemize}
    \item[$\mathbf{C}_\mu^\rho$] for every section $s$, the pushdown of $\mu$ is stochastically dominated by the $s$-marginal of $\rho$, i.e.~by the pushforward of $\rho$ by $x\longmapsto (x_{s(b)})_{b\in B}$.
\end{itemize}
Let us build an example where this condition is satisfied but where $\mu$ is not stochastically dominated by $\rho$.\\

Our sets $A$ and $B$ are defined by: $A:=\{1,2,3\}\times\{1,2,3\}$ and $B:=\{1,2,3\}$, and the function $\pi$ is the projection on the second coordinate. We will represent elements of $\{0,1\}^A$ and $\{0,1\}^B$ respectively by $3\times 3$ matrices with coefficients in $\{0,1\}$ and horizontal vectors of size $3$ with coefficients in $\{0,1\}$. 
Let $\mu$ be the uniform measure on the set given by these four matrices:
\[
\left[\begin{matrix}0&0&0\\0&0&0 \\0&0&0 \end{matrix}\right],\ 
\left[\begin{matrix}1&1&0\\0&0&0 \\0&0&0 \end{matrix}\right],\ 
\left[\begin{matrix}0&0&0\\1&0&1 \\0&0&0 \end{matrix}\right] \text{ and }
\left[\begin{matrix}0&0&0\\0&0&0 \\0&1&1 \end{matrix}\right].
\]
In particular, the pushdown of $\mu$ is the uniform measure the following vectors:
\[
\left[\begin{matrix}0&0&0\end{matrix}\right],\ 
\left[\begin{matrix}1&1&0\end{matrix}\right],\ 
\left[\begin{matrix}1&0&1 \end{matrix}\right] \text{ and }
\left[\begin{matrix}0&1&1 \end{matrix}\right].
\]
Similarly let $\rho$ be the uniform probability measure on these four matrices:
\[
\left[\begin{matrix}1&1&1\\1&1&1\\1&1&1\end{matrix}\right],\  \left[\begin{matrix}1&1&0\\1&0&1\\0&1&1\end{matrix}\right],\  \left[\begin{matrix}1&0&1\\0&1&1\\1&1&0\end{matrix}\right]\text{ and } \left[\begin{matrix}0&1&1\\1&1&0\\1&0&1\end{matrix}\right]. 
\]
Among these matrices, only the matrices 
\[
\left[\begin{matrix}1&1&1\\1&1&1\\1&1&1\end{matrix}\right],\text{ and }  \left[\begin{matrix}1&1&0\\1&0&1\\0&1&1\end{matrix}\right]
\]
are larger than either of the three matrices 
\[
\left[\begin{matrix}1&1&0\\0&0&0 \\0&0&0 \end{matrix}\right],\ 
\left[\begin{matrix}0&0&0\\1&0&1 \\0&0&0 \end{matrix}\right] \text{ and }
\left[\begin{matrix}0&0&0\\0&0&0 \\0&1&1 \end{matrix}\right].
\]
As a result, $\mu$ cannot be stochastically dominated by $\rho$. We still need to show that $\mathbf{C}_\mu^\rho$ holds. By symmetry, we may assume that $s(1)=1$. We still have nine possible sections to check, which is done in Figure~\ref{fig:table}. Every row corresponds to a given section and the coupling between $\mu$ and $\rho$ is given by which matrix is in which column.

\begin{rema}
Likewise, we can see that if we try to keep Assumptions~$\mathbf{A}_\rho$ and $\mathbf{B}_\mu^\rho$ and go beyond $[0,\infty]$-labels, we get stopped right away. Take labels to live in $\{0,1\}^2$, endowed with its product ordering. Let $B$ be a singleton and $A:=\{1,2\}$. Let $\mu$ be the uniform probability measure on the set given by these two configurations:
\[
\left[\begin{matrix}10\\00\end{matrix}\right] \text{ and }\left[\begin{matrix}00\\01\end{matrix}\right]
\]
and let $\rho$ be the uniform probability measure on the two configurations
\[
\left[\begin{matrix}10\\01\end{matrix}\right] \text{ and }\left[\begin{matrix}01\\10\end{matrix}\right].
\]
Then all assumptions hold but the conclusion does not. This minimal counterexample can be implemented in any label space that is not totally ordered and contains a global minimum --- which plays the role of the element $0$ of $[0,\infty]$.
\end{rema}

The following question remains open.

\begin{enonce}{Question}
Modify Theorem~\ref{main} as follows. Drop Assumptions~$\mathbf{A}_\rho~\mathbf{B}_\mu^\rho$ and assume instead that for every section $s$, the pushdown of $\mu$ is \emph{equal} to the $s$-marginal of $\rho$. Does this modified statement hold?
\end{enonce}

\newpage

\begin{figure}
~
\vspace{1.1cm}
\begin{center}
\begin{tabular}{||c | c c c c||} 
 \hline
 Section & $\left[\begin{matrix}1&1&0\end{matrix}\right]$ & $\left[\begin{matrix}1&0&1\end{matrix}\right]$ & $\left[\begin{matrix}0&1&1\end{matrix}\right]$  & $\left[\begin{matrix}0&0&0\end{matrix}\right]$ \\ [0.5ex] 
 \hline\hline 
$\begin{matrix} \  \\ \ \\ \ \\ \  \end{matrix}$ $\left[\begin{matrix} \textcolor{red}{*}& \textcolor{red}{*} & \textcolor{red}{*} \\ 0& 0 & 0 \\ 0 & 0 &0 \end{matrix}\right]$ $\begin{matrix} \  \\ \ \\ \ \\ \  \end{matrix}$ & $\left[\begin{matrix} \textcolor{red}{\textit{1}}& \textcolor{red}{\textit{1}} & \textcolor{red}{\textit{0}} \\ 1& 0 & 1 \\ 0 & 1 &1 \end{matrix}\right]$ & 
$\left[\begin{matrix} \textcolor{red}{\textit{1}}& \textcolor{red}{\textit{0}} & \textcolor{red}{\textit{1}} \\ 0& 1 & 1 \\ 1 & 1 &0 \end{matrix}\right]$ &
$\left[\begin{matrix} \textcolor{red}{\textit{1}}& \textcolor{red}{\textit{1}} & \textcolor{red}{\textit{1}} \\ 1& 1 & 1 \\ 1 & 1 &1 \end{matrix}\right]$ &
$\left[\begin{matrix} \textcolor{red}{\textit{0}}& \textcolor{red}{\textit{1}} & \textcolor{red}{\textit{1}} \\ 1& 1 & 0 \\ 1 & 0 &1 \end{matrix}\right]$ \\
 \hline 
$\begin{matrix} \  \\ \ \\ \ \\ \  \end{matrix}$ $\left[\begin{matrix}\textcolor{red}{*}& \textcolor{red}{*} & 0 \\ 0& 0 & \textcolor{red}{*} \\ 0 & 0 &0 \end{matrix}\right]$ $\begin{matrix} \  \\ \ \\ \ \\ \  \end{matrix}$ &  $\left[\begin{matrix}\textcolor{red}{\textit{1}}& \textcolor{red}{\textit{1}} & 0 \\ 1& 0 & \textcolor{red}{\textit{1}} \\ 0 & 1 &1 \end{matrix}\right]$ &
$\left[\begin{matrix}\textcolor{red}{\textit{1}}& \textcolor{red}{\textit{0}} & 1 \\ 0& 1 & \textcolor{red}{\textit{1}} \\ 1 & 1 &0 \end{matrix}\right]$ &
$\left[\begin{matrix}\textcolor{red}{\textit{1}}& \textcolor{red}{\textit{1}} & 1 \\ 1& 1 & \textcolor{red}{\textit{1}} \\ 1 & 1 &1 \end{matrix}\right]$ &
$\left[\begin{matrix}\textcolor{red}{\textit{0}}& \textcolor{red}{\textit{1}} & 1 \\ 1& 1 & \textcolor{red}{\textit{0}} \\ 1 & 0 &1 \end{matrix}\right]$ \\ 
 \hline
$\begin{matrix} \  \\ \ \\ \ \\ \  \end{matrix}$ $\left[\begin{matrix}\textcolor{red}{*}& \textcolor{red}{*} & 0 \\ 0& 0 & 0 \\ 0 & 0 &\textcolor{red}{*} \end{matrix}\right]$ $\begin{matrix} \  \\ \ \\ \ \\ \  \end{matrix}$ & 
$\left[\begin{matrix}\textcolor{red}{\textit{1}}& \textcolor{red}{\textit{1}} & 0 \\ 1& 0 & 1 \\ 0 & 1 &\textcolor{red}{\textit{1}} \end{matrix}\right]$ &
$\left[\begin{matrix}\textcolor{red}{\textit{1}}& \textcolor{red}{\textit{1}} & 1 \\ 1& 1 & 1 \\ 1 & 1 &\textcolor{red}{\textit{1}} \end{matrix}\right]$ &
$\left[\begin{matrix}\textcolor{red}{\textit{0}}& \textcolor{red}{\textit{1}} & 1 \\ 1& 1 & 0 \\ 1 & 0 &\textcolor{red}{\textit{1}} \end{matrix}\right]$ &
$\left[\begin{matrix}\textcolor{red}{\textit{1}}& \textcolor{red}{\textit{0}} & 1 \\ 0& 1 & 1 \\ 1 & 1 &\textcolor{red}{\textit{0}} \end{matrix}\right]$ \\
 \hline
$\begin{matrix} \  \\ \ \\ \ \\ \  \end{matrix}$ $\left[\begin{matrix}\textcolor{red}{*}& 0 & \textcolor{red}{*} \\ 0& \textcolor{red}{*} & 0 \\ 0 & 0 &0 \end{matrix}\right]$ $\begin{matrix} \  \\ \ \\ \ \\ \  \end{matrix}$ & $\left[\begin{matrix}\textcolor{red}{\textit{1}}& 0 & \textcolor{red}{\textit{1}} \\ 0& \textcolor{red}{\textit{1}} & 1 \\ 1 & 1 & 0 \end{matrix}\right]$ & 
$\left[\begin{matrix}\textcolor{red}{\textit{1}}& 1 & \textcolor{red}{\textit{1}} \\ 1& \textcolor{red}{\textit{1}} & 1 \\ 1 & 1 & 1 \end{matrix}\right]$ &
$\left[\begin{matrix}\textcolor{red}{\textit{0}}& 1 & \textcolor{red}{\textit{1}} \\ 1& \textcolor{red}{\textit{1}} & 0 \\ 1 & 0 & 1 \end{matrix}\right]$ &
$\left[\begin{matrix}\textcolor{red}{\textit{1}}& 1 & \textcolor{red}{\textit{0}} \\ 1& \textcolor{red}{\textit{0}} & 1 \\ 0 & 1 & 1 \end{matrix}\right]$ \\ 
 \hline
$\begin{matrix} \  \\ \ \\ \ \\ \  \end{matrix}$ $\left[\begin{matrix}\textcolor{red}{*}& 0 & 0 \\ 0& \textcolor{red}{*} & \textcolor{red}{*} \\ 0 & 0 &0 \end{matrix}\right]$ $\begin{matrix} \  \\ \ \\ \ \\ \  \end{matrix}$ & $\left[\begin{matrix}\textcolor{red}{\textit{1}}& 0 & 1\\ 0& \textcolor{red}{\textit{1}} & \textcolor{red}{\textit{1}} \\ 1 & 1 & 0 \end{matrix}\right]$ & 
$\left[\begin{matrix}\textcolor{red}{\textit{1}}& 1 & 0 \\ 1& \textcolor{red}{\textit{0}} & \textcolor{red}{\textit{1}} \\ 0 & 1 &1 \end{matrix}\right]$ &
$\left[\begin{matrix}\textcolor{red}{\textit{1}}& 1 & 1 \\ 1& \textcolor{red}{\textit{1}} & \textcolor{red}{\textit{1}} \\ 1 & 1 &1 \end{matrix}\right]$ &
$\left[\begin{matrix}\textcolor{red}{\textit{0}}& 1 & 1 \\ 1& \textcolor{red}{\textit{1}} & \textcolor{red}{\textit{0}} \\ 1 & 0 &1 \end{matrix}\right]$ \\
 \hline
$\begin{matrix} \  \\ \ \\ \ \\ \  \end{matrix}$ $\left[\begin{matrix}\textcolor{red}{*}& 0 & 0 \\ 0& \textcolor{red}{*} & 0 \\ 0 & 0 &\textcolor{red}{*} \end{matrix}\right]$ $\begin{matrix} \  \\ \ \\ \ \\ \  \end{matrix}$ & $\left[\begin{matrix}\textcolor{red}{\textit{1}}& 0 & 1 \\ 0& \textcolor{red}{\textit{1}} & 1 \\ 1 & 1 & \textcolor{red}{\textit{0}} \end{matrix}\right]$ & 
$\left[\begin{matrix}\textcolor{red}{\textit{1}}& 1 & 0 \\ 1& \textcolor{red}{\textit{0}} & 1 \\ 0 & 1 &\textcolor{red}{\textit{1}} \end{matrix}\right]$ &
$\left[\begin{matrix}\textcolor{red}{\textit{1}}& 1 & 1 \\ 1& \textcolor{red}{\textit{1}} & 1 \\ 1 & 1 &\textcolor{red}{\textit{1}} \end{matrix}\right]$ &
$\left[\begin{matrix}\textcolor{red}{\textit{0}}& 1 & 1 \\ 1& \textcolor{red}{\textit{1}} & 0 \\ 1 & 0 &\textcolor{red}{\textit{1}} \end{matrix}\right]$ \\
 \hline
$\begin{matrix} \  \\ \ \\ \ \\ \  \end{matrix}$ $\left[\begin{matrix}\textcolor{red}{*}& 0 & \textcolor{red}{*} \\ 0& 0 & 0 \\ 0 & \textcolor{red}{*} &0 \end{matrix}\right]$ $\begin{matrix} \  \\ \ \\ \ \\ \  \end{matrix}$ & $\left[\begin{matrix}\textcolor{red}{\textit{1}}& 1 & \textcolor{red}{\textit{0}} \\ 1& 0 & 1 \\ 0 & \textcolor{red}{\textit{1}} &1 \end{matrix}\right]$ & 
$\left[\begin{matrix}\textcolor{red}{\textit{1}}& 0 & \textcolor{red}{\textit{1}} \\ 0& 1 & 1 \\ 1 & \textcolor{red}{\textit{1}} &0 \end{matrix}\right]$ &
$\left[\begin{matrix}\textcolor{red}{\textit{1}}& 1 & \textcolor{red}{\textit{1}} \\ 1& 1 & 1 \\ 1 & \textcolor{red}{\textit{1}} &1 \end{matrix}\right]$ &
$\left[\begin{matrix}\textcolor{red}{\textit{0}}& 1 & \textcolor{red}{\textit{1}} \\ 1& 1 & 0 \\ 1 & \textcolor{red}{\textit{0}} &1 \end{matrix}\right]$ \\ 
 \hline
$\begin{matrix} \  \\ \ \\ \ \\ \  \end{matrix}$ $\left[\begin{matrix}\textcolor{red}{*}& 0 & 0 \\ 0& 0 & \textcolor{red}{*} \\ 0 & \textcolor{red}{*} &0 \end{matrix}\right]$ $\begin{matrix} \  \\ \ \\ \ \\ \  \end{matrix}$ & $\left[\begin{matrix}\textcolor{red}{\textit{1}}& 1 & 0 \\ 1& 0 & \textcolor{red}{\textit{1}} \\ 0 & \textcolor{red}{\textit{1}} &1 \end{matrix}\right]$ & 
$\left[\begin{matrix}\textcolor{red}{\textit{1}}& 0 & 1 \\ 0& 1 & \textcolor{red}{\textit{1}} \\ 1 & \textcolor{red}{\textit{1}} &0 \end{matrix}\right]$ &
$\left[\begin{matrix}\textcolor{red}{\textit{1}}& 1 & 1 \\ 1& 1 & \textcolor{red}{\textit{1}} \\ 1 & \textcolor{red}{\textit{1}} &1 \end{matrix}\right]$ &
$\left[\begin{matrix}\textcolor{red}{\textit{0}}& 1 & 1 \\ 1& 1 & \textcolor{red}{\textit{0}} \\ 1 & \textcolor{red}{\textit{0}} &1 \end{matrix}\right]$ \\
 \hline
$\begin{matrix} \  \\ \ \\ \ \\ \  \end{matrix}$ $\left[\begin{matrix}\textcolor{red}{*}& 0 & 0 \\ 0& 0 & 0 \\ 0 & \textcolor{red}{*} &\textcolor{red}{*} \end{matrix}\right]$ $\begin{matrix} \  \\ \ \\ \ \\ \  \end{matrix}$ & $\left[\begin{matrix}\textcolor{red}{\textit{1}}& 0 & 1 \\ 0& 1 & 1 \\ 1 & \textcolor{red}{\textit{1}} &\textcolor{red}{\textit{0}} \end{matrix}\right]$ & 
$\left[\begin{matrix}\textcolor{red}{\textit{1}}& 1 & 1 \\ 1& 1 & 1 \\ 1 & \textcolor{red}{\textit{1}} &\textcolor{red}{\textit{1}} \end{matrix}\right]$ &
$\left[\begin{matrix}\textcolor{red}{\textit{1}}& 1 & 0 \\ 1& 0 & 1 \\ 0 & \textcolor{red}{\textit{1}} &\textcolor{red}{\textit{1}} \end{matrix}\right]$ &
$\left[\begin{matrix}\textcolor{red}{\textit{0}}& 1 & 1 \\ 1& 1 & 0 \\ 1 & \textcolor{red}{\textit{0}} &\textcolor{red}{\textit{1}} \end{matrix}\right]$ \\
 \hline
 \hline
\end{tabular}
\end{center}
\caption{This table justifies that our construction satisfies Assumption~$\mathbf{C}_\mu^\rho$.}
\label{fig:table}
\end{figure}

\newpage

\section{Application of Theorem~\ref{lakon} to percolation}
\label{sec:perco}

The purpose of this section is to use Theorem~\ref{lakon} to reprove and generalise Theorem~\ref{thm:bs}.
Site percolation, bond percolation, and fibrations have been defined in the introduction. Recall that, given a graph $\mathscr G$, we write $V_\mathscr{G}$ for its set of vertices and $E_\mathscr{G}$ for its set of edges.

\subsection{Site percolation encompasses bond percolation}
\label{sec:bondtosite}

The bond-version of Theorem~\ref{thm:bs} follows from its site-version. Indeed, recall that given a countable locally finite graph $\mathscr G$, one can build another countable locally finite graph, $\mathscr{G}^\star$, by setting $V_{\mathscr{G}^\star}=E_{\mathscr G}$ and declaring two elements of $V_{\mathscr{G}^\star}$ to be adjacent if and only if the corresponding edges in $\mathscr G$ have exactly one vertex in common. Then, we have $p_c^\mathrm{site}(\mathscr{G}^\star)=p_c^\mathrm{bond}(\mathscr{ G})$.

It remains to take care of $\pi$. To do so, introduce $\La^\star_\pi$ defined as follows. Its vertex-set is given by the set of all edges of $\La$ that are mapped by $\pi$ to edges of $\Sm$. As for the edge-structure on $\La^\star_\pi$, it is simply that induced by $\La^\star$: in other words, two vertices of $\La^\star$ are adjacent if and only if, seen as edges of $\La$, they have exactly one vertex in common. Every fibration $\pi$ from $\La$ to $\Sm$ induces a fibration from $\La^\star_\pi$ to $\Sm^\star$. As $\La^\star_\pi$ is a subgraph of $\La^\star$, we have $p_c^\mathrm{site}(\La^\star)\le p_c^\mathrm{site}(\La^\star_\pi)$. Combining these observations with the equalities $p_c^\mathrm{site}(\La^\star)=p_c^\mathrm{bond}(\La)$ and $p_c^\mathrm{site}(\Sm^\star)=p_c^\mathrm{bond}(\Sm)$ reduces the case of bond percolation to that of site percolation.

\subsection{Measurability issues}
\label{sec:measurability}

To complete the proof of Theorem~\ref{thm:bs} via Theorem~\ref{lakon} sketched in the introduction, it remains to take care of the measurability of $\kappa$ and $\tilde{\kappa}$. There are two rather easy ways to take care of this: we may either reduce the problem to a finite setup where there is no measurability to check or indeed check the measurability we need.

\subsubsection{Reduction to a finite setup}\label{sec:meas-fin} Let $p\in[0,1]$ and let us consider $p$-site percolation on both $\La$ and $\Sm$. It suffices to prove that for every $x\in V_\La$, the probability that $x$ is connected to infinity in $\La$ is larger than or equal to the probability that $\pi(x)$ is connected to infinity in $\Sm$.
Therefore, it suffices to prove that, for every $n$, the probability that there is an open self-avoiding path\footnote{We say that a path is \defini{self-avoiding} if no vertex is visited more than once. We take the \defini{length} of a self-avoiding path to be the number of edges it uses, which is the number of vertices it visits minus one.} starting at $x$ of length $n$ is larger than or equal to the probability that there is an open self-avoiding path starting at $\pi(x)$ of length $n$. 
We are then able to conclude, as this inequality results from Proposition~\ref{finite}, with $A$ the $n$-ball centred at $x$ and $B$ the $n$-ball centred at $\pi(x)$.

We do not claim that $\pi$ induces a fibration from $A$ to $B$, simply that every self-avoiding path of length $n$ started at $\pi(x)$ --- which necessarily lies in $B$ --- can be lifted to a self-avoiding path of length $n$ in $x$ --- which necessarily lies in $A$.

\begin{rema}
	If we only care about Theorem~\ref{thm:bs}, we do not need to go through Section~\ref{sec:finite-to-gen}, as Proposition~\ref{finite} suffices.
\end{rema}
\subsubsection{Checking measurability}
\label{sec:meas-check}
We can enumerate all vertices of $\Sm$ and do the same for $\La$. This provides a well-ordering of these vertex-sets. Then, by lexicographic ordering, we get total orders on the set of $\mathbb{N}$-indexed self-avoiding paths in $\Sm$ and $\La$, respectively. These orders are such that any nonempty closed\footnote{for the topology of pointwise convergence, where vertex-sets are seen as discrete} set admits a smallest element. A percolation configuration being fixed, the set of \emph{open} $\mathbb{N}$-indexed self-avoiding paths is topologically closed. Likewise, some $\mathbb{N}$-indexed self-avoiding paths in $\Sm$ being fixed, the set of all its lifts is closed. Therefore, picking $\kappa$ to be the smallest open such path (when some infinite open path exists) and $\tilde{\kappa}$ to be the smallest lift of $\kappa$ provides a well-defined construction. The construction of $\kappa$ is indeed measurable because, by local finiteness, an open finite self-avoiding path admits an infinite open extension if and only if it admits finite open extensions of all sizes. Likewise, the construction of $\tilde\kappa$ is measurable: define $\tilde\kappa_0$ to be the smallest vertex in $\pi^{-1}(\kappa_0)$, then $\kappa_1$ to be the smallest neighbour of $\tilde\kappa_0$ in $\pi^{-1}(\kappa_1)$, etc.

\subsection{Generalisation of Theorem~\ref{thm:bs}}
\label{sec:genbs}

Actually, we can prove a more general result.

\begin{theo}
	\label{thm:strongbs}
	Let $\La$ and $\Sm$ be countable locally finite graphs. Let $\pi$ be a surjective map from $V_\La$ to $V_\Sm$. Assume that there is a measurable way to assign to every infinite self-avoiding path $\gamma:\mathbb{N}\to V_\Sm$ some infinite self-avoiding path $\tilde{\gamma}:\{n_\gamma,n_\gamma+1,\dots\}\to V_\La$ such that for every $m\ge n_\gamma$, we have $\pi(\tilde{\gamma}_m)=\gamma_m$.
	
	For every vertex $v$ in $V_\Sm$, let $p_v\in [0,1]$. On $\Sm$, consider the random configuration where each vertex is kept independently with probability $p_v$.
	
	Let $\mathfrak{X}$ be some random subset of $V_\La$. We assume that the family of random variables $(\mathfrak{X}\cap \pi^{-1}(v))_{v\in V_\Sm}$ is independent. Besides, for every $x\in V_\La$, assume that $\mathbb{P}(x\in \mathfrak{X})\ge p_v$.
	
	Then, the probability that there is an infinite path of retained vertices in $\La$ is larger than or equal to the the probability that there is an infinite path of retained vertices in $\Sm$.
\end{theo}

\begin{rema}
	\label{rem:gain}
	This theorem has two kinds of assumptions: the first paragraph is of geometric nature while the second and third paragraphs are probabilistic. The techniques of \cite{bs} can be used to handle this level of probabilistic generality: we can get a classical proof by exploration of the above theorem if the geometric assumption is replaced by the assumption that there is a fibration from $\La$ to $\Sm$. However, the techniques of \cite{bs} really need the stronger geometric condition of having a fibration: we need to be able to continue lifting our path wherever the past exploration led us, which imposes us to ask for the possibility to lift edges $\{u,v\}$ to \emph{all} $x\in\pi^{-1}(u)$. Theorem~\ref{thm:strongbs} relaxes this condition, only asking for every infinite self-avoiding path to admit at least one lift, which may start where it wants to (and can be measurably picked). The condition is actually relaxed further, by allowing to forget the beginning of the path.
\end{rema}

\begin{proof}
	Apply our new proof of Theorem~\ref{thm:bs} but use Corollary~\ref{coro:indep} instead of Theorem~\ref{lakon}. Regarding the measurability of $\kappa$, the argument of Section~\ref{sec:meas-fin} does not apply anymore but that of Section~\ref{sec:meas-check} still applies. As for $\tilde\kappa$, we assumed the measurability we need.
\end{proof}

As explained in Remark~\ref{rem:gain}, the main gain in Theorem~\ref{thm:strongbs} is about geometry. Convincing examples appear to be lacking so far: in the situations we came up with, we could use a few tricks to get back to the case of fibrations. Still, we find conceptual value in not having to \emph{assume} that we are in the fibration setup or to use ad hoc tricks. This conceptual value says more about the nature of our arguments and our understanding than about the list of concrete examples we can handle.

Regarding the probabilistic gain, we indeed extend the scope of Theorem~\ref{thm:bs}, but in a way that Benjamini and Schramm could have covered with their techniques if they wanted to. It is easy to come up with diverse examples. Let us provide one that we find natural and interesting.

\begin{enonce}[remark]{Example}
    \label{example}
    Assume that every $\pi^{-1}(x)$ has cardinality 2 and that all $p_v$'s are equal to $\sfrac12$. Independently, for each $x\in V_\Sm$, prescribe $\mathfrak{X}\cap \pi^{-1}(x)$ to be a uniformly chosen singleton included in $\pi^{-1}(x)$. In other words, the percolation model on $\La$ is defined by picking independently and uniformly in each $\pi$-fibre one vertex that we keep, letting the other aside.
\end{enonce}

Using Theorem~\ref{main} instead of Corollary~\ref{coro:indep} would lead to an even more general version of Theorem~\ref{thm:strongbs}. Using Corollary~\ref{coro:exchange} would lead to a variation of it, and other examples. A priori, the probabilistic gain in this generalisation and this variation cannot be obtained by direct adaptation of \cite{bs}, as their exploration argument relies heavily on independence.

Let us conclude Section~\ref{sec:genbs} by an amusing observation. Let $\Sm$ be a countable locally finite graph. We define the graph $\La$ to be a collection of disjoint, noninteracting infinite rays, one per $\mathbb{N}$-indexed self-avoiding path in $\Sm$. There is a surjective map from $V_\La$ to $V_\Sm$ satisfying the condition of Theorem~\ref{thm:strongbs}. Therefore, by Theorem~\ref{thm:strongbs}, it seems that $p_c^\mathrm{site}(\Sm)\ge p_c^\mathrm{site}(\La)$. But we have $p_c^\mathrm{site}(\La)=1$, as it consists of a bunch of rays and each ray has $p_c=1$. We seem to get into a contradiction, namely that every $\Sm$ satisfies $p_c^\mathrm{site}(\Sm)=1$. It turns out that we glossed over a single but important detail: for this argument to be sound, we need the vertex-set of $\La$ to be countable, which seldom happens. For sure, whenever $p_c^\mathrm{site}(\Sm)<1$, the set of its $\mathbb{N}$-indexed self-avoiding paths is uncountable: this fact is very easy but a convoluted proof is given by the proof by contradiction of the present paragraph. In some loose sense, this remark recalls that percolation does not behave well if the ways to reach infinity are decomposed into uncountably many types, and conversely Theorem~\ref{thm:strongbs} says that it does behave well for countably many types.

\subsection{A question}
For every $n\ge 0$, let $\mathscr{I}_n$ denote the graph with vertex-set $\{0,1,\dots,n\}$ and with edges corresponding to $\{i,j\}$ such that $|i-j|=1$. Likewise, for $n\ge 3$, let $\mathscr{C}_n$ denote the $n$-cycle, i.e. the graph with vertex-set $\mathbb{Z}/n\mathbb{Z}$ and with edges corresponding to $\{i,j\}$ such that $j=i\pm1$.
Also, given two graphs $\mathscr{G}$ and $\mathscr{H}$, define the product $\mathscr{G}\times\mathscr{H}$ by letting $V_{\mathscr{G}\times\mathscr{H}}:=V_\mathscr{G}\times V_\mathscr{H}$ and declaring $(u,x)$ to be adjacent to $(v,y)$ if either ``$u=v$ and $y$ is a neighbour of $x$'' or ``$x=y$ and $v$ is a neighbour of $u$''. In this section, $p_c$ can be taken to mean either always $p_c^\mathrm{site}$ or always $p_c^\mathrm{bond}$, both leading to a correct reading.

Let $\mathscr{G}$ be a countable locally finite graph. By graph-inclusion, it is clear that \[
m\ge n\implies p_c(\mathscr{G}\times\mathscr{I}_m)\le p_c(\mathscr{G}\times\mathscr{I}_n).
\]
Because of Theorem~\ref{thm:bs}, it is also the case that for every $m,n\ge3$, whenever $m$ is a multiple of $n$, we have $p_c(\mathscr{G}\times\mathscr{C}_m)\le p_c(\mathscr{G}\times\mathscr{C}_n)$.

\begin{enonce}{Question}
    Is it the case that for every countable locally finite graph $\mathscr{G}$ and every $m\ge n\ge 3$, we have $p_c(\mathscr{G}\times\mathscr{C}_m)\le p_c(\mathscr{G}\times\mathscr{C}_n)$?
\end{enonce}

This question asks whether the monotonicity that holds for the ``slab'' $\mathscr{G}\times\mathscr{I}_n$ also holds for its ``periodised'' version $\mathscr{G}\times\mathscr{C}_n$.

\section{Revisiting the BK inequality}
\label{sec:bk}

Let $C$ be a finite set. A subset $\mathcal{E}\subset \{0,1\}^C$ is said to be \defini{increasing} if its indicator function is weakly increasing for the product order. In other words, $\mathcal{E}$ is increasing if and only if for every $\omega\in\mathcal{E}$ and every $\omega'\in\{0,1\}^C$ satisfying $\forall i,\,\omega_i\le \omega'_i$, we have $\omega'\in\mathcal{E}$.

Given two subsets $\mathcal{E}_1$ and $\mathcal{E}_2$ of $\{0,1\}^C$, the \defini{disjoint occurrence} of $\mathcal{E}_1$ and $\mathcal{E}_2$ is denoted by $\mathcal{E}_1\circ\mathcal{E}_2$ and defined as follows. It is the set of all $\omega\in\{0,1\}^C$ such that there are disjoint subsets $P_1$ and $P_2$ of $C$ such that the following two conditions hold:
\begin{itemize}
    \item for every $\omega'\in\{0,1\}^C$ such that $\omega'$ and $\omega$ agree on their restriction to $P_1$, we have $\omega'\in \mathcal{E}_1$,
    \item for every $\omega'\in\{0,1\}^C$ such that $\omega'$ and $\omega$ agree on their restriction to $P_2$, we have $\omega'\in \mathcal{E}_2$.
\end{itemize}
In other words, we can find two disjoint ``witnesses'', one guaranteeing that $\mathcal{E}_1$ holds and the other one doing the same for $\mathcal{E}_2$. 

The BK inequality, due to van den Berg and Kesten \cite{bk}, is very useful in percolation theory and goes as follows.

\begin{theo}[BK inequality]
Let $C$ be a finite set. Let $\mathbb{P}$ be a probability measure on $\{0,1\}^C$ of the form $\bigotimes_i \mathrm{Bernoulli}(p_i)$. Let $\mathcal{E}_1$ and $\mathcal{E}_2$ be two increasing subsets of $\{0,1\}^C$.

Then, we have $\mathbb{P}(\mathcal{E}_1\circ \mathcal{E}_2)\le \mathbb{P}(\mathcal{E}_1)\mathbb{P}(\mathcal{E}_2)$.
\end{theo}

A typical application of the BK inequality in percolation theory goes as follows. Let $o$ be a vertex in some graph, and let $n\ge 1$. Perform bond percolation of some fixed parameter $p$. Denote by $a_1(n,p)$ the probability that there is some open path starting at $o$ that touches the sphere of radius $n$ centred at $o$. Denote by $a_2(n,p)$ the probability that there are two \emph{edge-wise disjoint} open paths starting at $o$ that touch the sphere of radius $n$ centred at $o$. As the existence of an open path from the origin to the $n$-sphere is an increasing event, the BK inequality yields $a_2(n,p)\ge a_1(n,p)^2$. Therefore, any lower bound on $a_1(n,p)$ automatically yields some lower bound for $a_2(n,p)$. Getting sharper inequalities than $a_2(n,p)\ge a_1(n,p)^2$ requires hard work; see \cite{bn,gm,rv}.

\begin{proof}[Proof of the BK inequality given our results]
For now, let us assume that all $p_i$ are equal to a fixed value $p\in[0,1]$, which is the case we are usually interested in percolation theory.
Let $A:=C\times\{1,2\}$, $B:=C$ and $\pi:A\to B$ be the map $(i,j)\longmapsto i$. Endow $\{0,1\}^B$ with the probability measure  $\mathbb{P}=\mathrm{Bernoulli}(p)^{\otimes B}$. For every $i\in B$, let us define $X_i:\{0,1\}^B\to \{0,1\}$ by $X_i(\omega)= \omega_i$. At last, let $\mathbf{P}:=\mathrm{Bernoulli}(p)^{\otimes A}$.

Let us pick a random section $S$ as follows. Whenever $X\in\mathcal{E}_1\circ\mathcal{E}_2$, pick disjoint witnesses $P_1$ and $P_2$ and make sure that $S$ maps every element of $P_1$ to 1 and every element of $P_2$ to 2 --- there is no constraint outside $P_1\cup P_2$. Such a choice is possible as $P_1$ and $P_2$ are disjoint. 
When $X\notin\mathcal{E}_1\circ\mathcal{E}_2$, the random section $S$ can be arbitrarily chosen. Let $Y_{i}^{(j)}:=X_{i}\,\mathds{1}_{S(i)=(i,j)}$.

Let $\mathcal{E}_\mathrm{lift}$ be the set of all $(\omega^{(1)},\omega^{(2)})\in \{0,1\}^A$ such that $\omega^{(1)}\in \mathcal{E}_1$ and $\omega^{(2)}\in\mathcal{E}_2$. This event is increasing, as $\mathcal{E}_1$ and $\mathcal{E}_2$ are themselves increasing. 
By construction, we have $\mathcal{E}_1\circ\mathcal{E}_2\subset \{(Y^{(1)},Y^{(2)})\in\mathcal{E}_\mathrm{lift}\}$, so that $\mathbb{P}(\mathcal{E}_1\circ\mathcal{E}_2)\leq\mathbb{P}((Y^{(1)},Y^{(2)})\in\mathcal{E}_\mathrm{lift})$. Because of Theorem~\ref{lakon} and as $\mathcal{E}_\mathrm{lift}$ is increasing, we have $\mathbb{P}((Y^{(1)},Y^{(2)})\in\mathcal{E}_\mathrm{lift})\le \mathbf{P}(\mathcal{E}_{\mathrm{lift}})$. By independence, we have $\mathbf{P}(\mathcal{E}_{\mathrm{lift}})=\mathbb{P}(\mathcal{E}_1)\mathbb{P}(\mathcal{E}_2)$. We conclude by composing the inequalities:
\[
\mathbb{P}(\mathcal{E}_1\circ\mathcal{E}_2)\leq\mathbb{P}((Y^{(1)},Y^{(2)})\in\mathcal{E}_\mathrm{lift})\le \mathbf{P}(\mathcal{E}_{\mathrm{lift}})=\mathbb{P}(\mathcal{E}_1)\mathbb{P}(\mathcal{E}_2).
\]
To get the version with different $p_i$, it suffices to have a version of Theorem~\ref{lakon} with different $p_i$. Corollary~\ref{coro:indep-va} is more than sufficient to this end.
\end{proof}

As Corollary~\ref{coro:indep-va} and Theorem~\ref{main} are much more general than Theorem~\ref{lakon}, we can obtain variations of the BK inequality. For instance, with the same proof as before but using Corollary~\ref{coro:indep-va} instead of Theorem~\ref{lakon}, we get the following result. It may well not be new but we currently do not know of a place where such a statement is written.

\begin{prop}
\label{prop:bk}
Let $C$ be a finite set. Let $\mathbb{P}$ be a probability measure on $\{0,1\}^C$ of the form $\bigotimes_i \mathrm{Bernoulli}(p_i)$. Let $\mathcal{E}_1$ and $\mathcal{E}_2$ be two increasing subsets of $\{0,1\}^C$. Let $\mathbf{P}$ be a probability measure on $\{0,1\}^{C\times\{1,2\}}$ of the form $\bigotimes _{i\in C}\rho_i$, where $\rho_i$ is a probability measure on $\{0,1\}^2$ such that both its marginals are Bernoulli of parameter at least $p_i$.

Then, we have $\mathbb{P}(\mathcal{E}_1\circ \mathcal{E}_2)\le \mathbf{P}(\mathcal{E}_1\times\mathcal{E}_2)$.
\end{prop}

For instance, applying Proposition~\ref{prop:bk} in the situation of Example~\ref{example} with $p_i=\sfrac12$, we recover Theorem~3 of \cite{bn}, under the additional assumptions that we consider increasing events. Beffara and Nolin attribute their Theorem~3 to Reimer \cite[Theorem~1.2]{r}.

\begin{rema}
Reimer proved that the BK inequality holds without assuming that the events under consideration are increasing \cite{r}. We cannot use our techniques to prove this stronger inequality. Indeed, it is crucial that $\mathcal{E}_\mathrm{lift}$ is increasing in order to obtain the inequality $\mathbb{P}((Y^{(1)},Y^{(2)})\in\mathcal{E}_\mathrm{lift})\le \mathbf{P}(\mathcal{E}_{\mathrm{lift}})$. To get Reimer's inequality, it would suffice to have a lifting theorem which, in the phrasing of the abstract, also enables Alice to lift zeroes in the same way she lifts ones, without Bob being able to overwrite. However, this ``theorem'' fails badly, even for $|A|=2$ and $|B|=1$.
\end{rema}

\section{Revisiting strict monotonicity of $p_c$ with respect to quotient}
\label{sec:strict}

In \cite{ms}, Martineau and Severo proved a result regarding \emph{strict} monotonicity of $p_c$ for bond percolation and site percolation. This answered in particular Question~1 of \cite{bs}. In this section, we provide a new proof of this result for the case of bond percolation.

This proof does not rely on the new coupling results of other sections but it still has a very strong scent of coupling. Contrary to \cite{ms}, our proof does not use the theory of essential enhancements. This section shares with \cite{v} the philosophy of studying percolation with couplings instead of differential inequalities.
Before stating the main result of this section, namely Theorem~\ref{thm:strict}, let us introduce some relevant definitions.

Given a graph $\mathscr{G}$, a \defini{subtree} is a graph $\mathscr{T}$ that is a tree and satisfies $V_\mathscr{T}\subset V_\mathscr{G}$ and $E_\mathscr{T}\subset E_\mathscr{G}$.
Let $\La$ and $\Sm$ denote two countable locally finite graphs. A map $\pi:V_\La\to V_\Sm$ satisfies the \defini{disjoint tree-lifting property} if for any two distinct vertices $x_1$ and $x_2$ in $\La$ satisfying $\pi(x_1)=\pi(x_2)$, for every subtree $\mathscr{T}$ of $\Sm$ containing the vertex $\pi(x_1)$, there are subtrees $\mathscr{T}_1$ and $\mathscr{T}_2$ of $\La$ such that:
\begin{itemize}
    \item the vertex-sets of $\mathscr{T}_1$ and $\mathscr{T}_2$ are disjoint,
    \item for every $i\in\{1,2\}$, the vertex $x_i$ belongs to $\mathscr{T}_i$,
    \item for every $i\in\{1,2\}$, the map $\pi$ induces a well-defined graph isomorphism from $\mathscr{T}_i$ to $\mathscr{T}$.
\end{itemize}

In this paper, a \defini{cycle} in a graph $\mathscr{G}$ is a finite sequence of vertices $(v_0,\dots,v_\ell)$ such that each vertex is adjacent \emph{or equal} to the previous one, and such that $v_\ell$ is equal to $v_0$. The number $\ell$ is called the \defini{length} of the cycle. Given $\pi:V_\La\to V_\Sm$, say that \defini{we can switch floors by lifting short cycles} if there is a constant $c$ such that the following holds: for every vertex $x$ in $\La$, there is a distinct vertex $y$ in $\La$ satisfying $\pi(x)=\pi(y)$ and such that there is a cycle of length at most $c$ in $\Sm$ that admits a lift starting at $x$ and ending at $y$. In the previous sentence, a lift of a cycle $(v_0,\dots,v_\ell)$ is a path $(x_0,\dots,x_\ell)$ in $\La$ such that, for every $i\le \ell$, we have $\pi(x_i)=v_i$. We allow consecutive vertices of a cycle to be equal because we want Theorem~\ref{thm:strict} to cover situations such as $\La$ being the cubic lattice, $\Sm$ the square lattice and $\pi:(x,y,z)\longmapsto (x,y)$.

At last, recall that a graph has \defini{bounded degree} if there is some constant $D$ such that each vertex has at most $D$ neighbours.

\begin{theo}[Martineau--Severo, 2019]
\label{thm:strict}
Let $\La$ and $\Sm$ be two countable locally finite graphs. Assume that $\Sm$ has bounded degree. Further, assume that there is a surjective map $\pi:V_\La\to V_\Sm$ satisfying the disjoint tree-lifting property and such that we can switch floors by lifting short cycles.

If $p_c^\mathrm{bond}(\La)<1$, then we have $p_c^\mathrm{bond}(\La)<p_c^\mathrm{bond}(\Sm)$.
\end{theo}

\begin{rema}
    It would be natural to further assume that $\pi$ is a fibration. One can put this condition in the theorem or not without changing the scope of the result, as any $\pi$ satisfying the conditions of Theorem~\ref{thm:strict} is necessarily a fibration: any time we want to lift an edge to some $x_1\in V_\La$, we can use floor-switching to find a distinct $x_2\in V_\La$ such that $\pi(x_1)=\pi(x_2)$, and then apply the weak disjoint-tree lifting property to $(x_1,x_2)$ and a tree consisting of a single edge. The fact that $\pi$ is a fibration will indeed be used in Section~\ref{sec:last-step}.
\end{rema}

\begin{rema}\label{rem:compawithms}
The statement we give here does not correspond exactly to \cite[Theorem~1.1]{ms}. Our Theorem~\ref{thm:strict} directly implies theirs (for bond percolation only). The techniques of \cite{ms} actually prove the above statement for both site and bond percolation, even though they did not state it in this way. Theorem~1.1 of \cite{ms} corresponds to further assuming that $\La$ has bounded degree, that $\pi$ is 1-Lipschitz, and that $\La$ and $\Sm$ are nonempty and connected. Apart from the connectedness of $\Sm$, these additional assumptions are actually not used in the proof of \cite{ms}. Connectedness of $\Sm$ is really used. However, it could be removed, as in our proof, by noticing that the bounds obtained in \cite{ms} are not mere strict inequalities but are in some sense ``uniform''.
\end{rema}

\begin{rema}
We decided not to assume connectedness of the graphs for two reasons. First, obtaining such a result is strictly more informative and captures some ``uniformity'' of the strict inequality: if we have sequences of connected graphs $(\La_n)$ and $(\Sm_n)$ such that $(\La_n,\Sm_n)$ satisfies the assumptions of Theorem~\ref{thm:strict} with bounds on the degree and on the $c$ of ``short cycles'' that are uniform in $n$, then our result applied to $\bigsqcup_n \La_n$ and  $\bigsqcup_n \Sm_n$ guarantees that $\inf_n p_c^\mathrm{bond}\left(\La_n\right)< \inf_n p_c^\mathrm{bond}\left(\Sm_n\right)$. If we assumed connectedness in Theorem~\ref{thm:strict}, we would know that, for every $n$, the inequality $p_c^\mathrm{bond}\left(\La_n\right)<  p_c^\mathrm{bond}\left(\Sm_n\right)$ holds but we would not be able to derive strict inequality for the infima.

A second reason not to assume connectedness is that the setup we consider is well behaved if we want to transfer knowledge from site-percolation to bond-percolation: see Section~\ref{sec:bondtosite}. Even when $\Sm$ and $\La$ are connected, the graph $\La^\star_\pi$ may not be connected, and our point is that it is not necessary to find a connected substitute for $\La^\star_\pi$. In the end, this second reason is not concretely used here, as we do not prove the site-version but only the bond-version of \cite[Theorem~1.1]{ms}. In other words, we have decided to work in a good framework, even though we do not use its goodness.
\end{rema}

\subsubsection*{Informal overview of the proof strategy}\label{strat} As in \cite{ms}, we want to introduce some sort of new percolation model on $\Sm$ (or a closely related graph) such that we have $p_c(\La)\le p_c^\mathrm{new}(\Sm)<p_c(\Sm)$. To do so, we will first deterministically partition $\Sm$ (or something closely related to it) into blocks, which will be called ``cells''. Then, we will take the usual percolation and allow some ``bonus'' inside each cell, and the whole game is to do so in such a way that at the same time:
\begin{enumerate}
    \item this bonus can be ``implemented, via $\pi$, in the usual percolation on $\La$'',
    \item this bonus is sufficiently strong for the inequality $p_c^\mathrm{new}(\Sm)<p_c(\Sm)$ to be true (and provable).
\end{enumerate}

A noteworthy feature of this proof is that each of the two inequalities $p_c(\La)\le p_c^\mathrm{new}(\Sm)$ and $p_c^\mathrm{new}(\Sm)<p_c(\Sm)$ is proved by using some exploration algorithm but that the algorithms are different --- one specific for each inequality. There is no need for any form of ``compatibility'' between these algorithms, apart from correctly describing our percolation process.

\newcommand{\Pp}{\mathbb{P}}
\newcommand{\cell}{\mathcal{C}}
\newcommand{\precell}{\tilde{C}}
\newcommand{\centres}{\mathfrak{C}}
\newcommand{\centers}{\centres}
\newcommand{\Lap}{\mathscr{L}'}
\newcommand{\Smap}{\mathscr{S}'}
\newcommand{\super}{K^+_{p,s}}
\newcommand{\bound}{\mathfrak{B}}
\newcommand{\pcaug}{p_c^\mathrm{aug}}
\newcommand{\minconfig}{\mathsf{empty}}
\newcommand{\maxconfig}{\mathsf{full}}

\subsection{Step 1: reduction of the problem}

To begin with, let us explain why it suffices to prove $p_c(\La)<p_c(\Sm)$ in the following setup.

\begin{enonce}[remark]{Notation}
    Given a graph $\mathscr{G}$, for $A\subset V_\mathscr{G}$, we denote by $E_A$ the set of all edges of $\mathscr{G}$ the two endpoints of which belong to $A$. We define the vertex-boundary $\partial_V A$ to be the set of all vertices in $A$ having a neighbour outside $A$. The interior of $A$ is defined as $\mathring{A}:=A\setminus\partial_V A$. The edge-boundary of $A$ is $\partial_E A=\{e\in E_\mathscr{G}\,:\,|e\cap A|=1\}$. The notation $B_u(\alpha)$ stands for the closed ball of radius $\alpha$ centred at $u$, i.e.~$B_u(\alpha)=\{v\in V_\mathscr{G}\,:\,d(u,v)\le \alpha\}.$
\end{enonce}

\begin{enonce*}[remark]{The setup}
    The graphs $\La$ and $\Sm$ have nonempty countable vertex-sets. The graph $\La$ is locally finite. The graph $\Sm$ has bounded degree and satisfies $0<p_c^\mathrm{bond}(\Sm)<1$. We have a surjective map $\pi:V_\La\to V_\Sm$ with the disjoint tree-lifting property. We also have positive integers $r$ and $R$ such that the following holds. For every vertex $x$ in $\La$, there is a distinct vertex $y$ in $\La$ satisfying $\pi(x)=\pi(y)$ and such that there is a cycle of length at most $2r$ in $\Sm$ that admits a lift starting at $x$ and ending at $y$. Furthermore, we have a set of vertices $\centres\subset V_\Sm$ such that, for every $u\in\centres$, there exists a set $\cell_u\subset V_\Sm$ such that the following conditions hold:
\begin{enumerate}
    \item \label{item:partition} when $u$ ranges over $\centres$, the $E_{\cell_u}$'s form a partition of $E_{\Sm}$,
    \item for every $u\in\centres$, the interior of $\cell_u$ is connected,
    \item for every $u\in \centres$, we have the inclusions $B_u(r)\subset\mathring\cell_u$ and $\cell_u\subset B_u(R)$.
\end{enumerate}
\end{enonce*}

We start with $(\La_0,\Sm_0,\pi_0)$ satisfying the assumptions of the theorem. As $\Sm_0$ has bounded degree, it has $p_c^\mathrm{bond}(\Sm_0)>0$. Without loss of generality, assume that $p_c^\mathrm{bond}(\Sm_0)<1$, as the statement is trivial when $p_c^\mathrm{bond}(\Sm_0)=1$. In particular, $V_{\Sm_0}$ is nonempty and so is $V_{\La_0}$, by surjectivity of $\pi$. We can pick $r_0\ge1$ such that for every vertex $x$ in $\La_0$, there is a distinct vertex $y$ in $\La$ satisfying $\pi(x)=\pi(y)$ and such that there is a cycle of length at most $2r_0+1$ in $\Sm$ that admits a lift starting at $x$ and ending at $y$. Removing edges from $\La_0$ cannot decrease its $p_c^\mathrm{bond}$, so we may remove from it all edges that are not mapped, via $\pi_0$, to edges of $\Sm_0$. Notice that doing so preserves the assumptions made on $\pi_0$, which are only concerned by how $\pi_0$ lifts edges of $\Sm_0$. In other words, without loss of generality, we may and will assume that $\pi_0$ induces a well defined map from $E_{\La_0}$ to $E_{\Sm_0}$.  We want to prove that $p_c^\mathrm{bond}(\La_0)<p_c^\mathrm{bond}(\Sm_0)$.

Let $\centres$ denote a maximal $(2r_0+1)$-separated subset of $V_{\Sm_0}$, i.e.~a set any two distinct elements of which are at distance at least $2r_0+1$ and that is maximal with respect to inclusion among all such sets. Fix, for every $u\in\centers$, a set $\precell_u \subset V_{\Sm_0}$ in such a way that the following conditions hold:
\begin{enumerate}
    \item when $u$ ranges over $\centres$, the $\precell_u$'s form a partition of $V_{\Sm_0}$,
    \item\label{item:connected-interior} every $\precell_u$ is connected,
    \item for every $u\in \centres$, we have the inclusions $B_u(r_0)\subset \precell_u\subset B_u(2r_0)$.
\end{enumerate}
Such a choice is indeed possible: for instance, one can take the Voronoi cells associated with $\centres$ and tie break arbitrarily all cases of equality --- see Figure~\ref{fig:cells}.

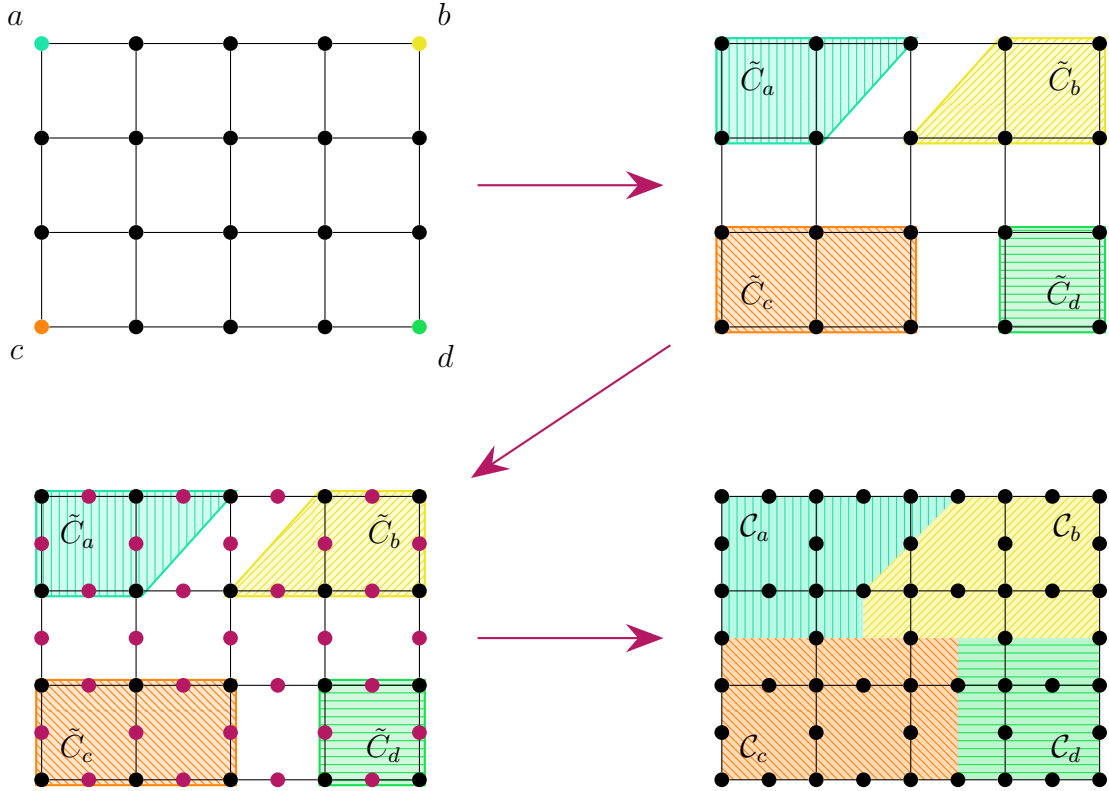
\begin{figure}[h]
\centering

\begin{tikzpicture}
\matrix(base)[matrix of nodes, nodes in empty cells, nodes={circle, fill = black, scale=.5}, column sep = 30pt, row sep = 30pt]
{
 \node(base-1-1)[fill=u1]{};& & & &\node(base-1-5)[fill=u2]{};\\
 & & & & \\
 & & & & \\
 \node(base-4-1)[fill=u3]{};& & & &\node(base-4-5)[fill=u4]{};\\
};
\foreach \a in {1, ..., 4}{
\draw (base-\a-1)-- (base-\a-2);
\draw (base-\a-2)-- (base-\a-3);
\draw (base-\a-3)-- (base-\a-4);
\draw (base-\a-4)-- (base-\a-5);
}
\foreach \a in {1, ..., 5}{
\draw (base-1-\a)-- (base-2-\a);
\draw (base-2-\a)-- (base-3-\a);
\draw (base-3-\a)-- (base-4-\a);
}
\node[above left=1pt of base-1-1, text=black]{$a$};
\node[above right=1pt of base-1-5, color=u2, text=black]{$b$};
\node[below left=1pt of base-4-1, color=u3, text=black]{$c$};
\node[below right=1pt of base-4-5, color = u4, text= black]{$d$};

\matrix(voron)[right=100pt of base, matrix of nodes, nodes in empty cells, nodes={circle, fill = black, scale=.5}, column sep = 30pt, row sep = 30pt]
{
 & & & & \\
 & & & & \\
 & & & & \\
 & & & & \\
};
\foreach \a in {1, ..., 4}{
\draw (voron-\a-1)-- (voron-\a-2);
\draw (voron-\a-2)-- (voron-\a-3);
\draw (voron-\a-3)-- (voron-\a-4);
\draw (voron-\a-4)-- (voron-\a-5);
}
\foreach \a in {1, ..., 5}{
\draw (voron-1-\a)-- (voron-2-\a);
\draw (voron-2-\a)-- (voron-3-\a);
\draw (voron-3-\a)-- (voron-4-\a);
}
\node[below right=1pt of voron-1-1]{$\precell_{a}$};
\node[below left=1pt of voron-1-5]{$\precell_{b}$};
\node[above right=1pt of voron-4-1]{$\precell_{c}$};
\node[above left=1pt of voron-4-5]{$\precell_{d}$};

\begin{scope}[on background layer]
\filldraw[preaction={fill, u1!20}, draw=u1,pattern={vertical lines}, pattern color=u1,  thick] (voron-1-1.north west)--(voron-1-3.north east)--(voron-2-2.south east)--(voron-2-1.south west)--(voron-1-1.north west);
\end{scope}

\begin{scope}[on background layer]
\filldraw[preaction={fill, u2!20},pattern = {north east lines}, pattern color = u2, draw=u2, thick] (voron-1-4.north west)--(voron-1-5.north east)--(voron-2-5.south east)--(voron-2-3.south west)--(voron-1-4.north west);
\end{scope}

\begin{scope}[on background layer]
\filldraw[preaction={fill, u3!20}, pattern={north west lines}, pattern color=u3, draw=u3, thick] (voron-3-1.north west)--(voron-3-3.north east)--(voron-4-3.south east)--(voron-4-1.south west)--(voron-3-1.north west);
\end{scope}

\begin{scope}[on background layer]
\filldraw[preaction={fill, u4!20}, draw=u4, pattern = {horizontal lines}, pattern color = u4,  thick] (voron-3-4.north west)--(voron-3-5.north east)--(voron-4-5.south east)--(voron-4-4.south west)--(voron-3-4.north west);
\end{scope}

\matrix(double)[below=50pt of base, matrix of nodes, nodes in empty cells, nodes={circle, fill = black, scale=.5}, column sep = 30pt, row sep = 30pt]
{
 & & & & \\
 & & & & \\
 & & & & \\
 & & & & \\
};
\foreach \a in {1, ..., 4}{
\draw (double-\a-1)-- (double-\a-2) node [midway, circle, fill = redgentil, scale=.5]{};
\draw (double-\a-2)-- (double-\a-3) node [midway, circle, fill = redgentil, scale=.5]{};
\draw (double-\a-3)-- (double-\a-4) node [midway, circle, fill = redgentil, scale=.5]{};
\draw (double-\a-4)-- (double-\a-5) node [midway, circle, fill = redgentil, scale=.5]{};
}
\foreach \a in {1, ..., 5}{
\draw (double-1-\a)-- (double-2-\a) node [midway, circle, fill = redgentil, scale=.5]{};
\draw (double-2-\a)-- (double-3-\a) node [midway, circle, fill = redgentil, scale=.5]{};
\draw (double-3-\a)-- (double-4-\a) node [midway, circle, fill = redgentil, scale=.5]{};
}
\node[below right=1pt of double-1-1]{$\precell_{a}$};
\node[below left=1pt of double-1-5]{$\precell_{b}$};
\node[above right=1pt of double-4-1]{$\precell_{c}$};
\node[above left=1pt of double-4-5]{$\precell_{d}$};

\begin{scope}[on background layer]
\filldraw[preaction={fill, u1!20}, draw=u1, pattern={vertical lines}, pattern color=u1,  thick] (double-1-1.north west)--(double-1-3.north east)--(double-2-2.south east)--(double-2-1.south west)--(double-1-1.north west);
\end{scope}

\begin{scope}[on background layer]
\filldraw[preaction={fill, u2!20},pattern = {north east lines}, pattern color = u2, draw=u2, thick] (double-1-4.north west)--(double-1-5.north east)--(double-2-5.south east)--(double-2-3.south west)--(double-1-4.north west);
\end{scope}

\begin{scope}[on background layer]
\filldraw[preaction={fill, u3!20}, pattern={north west lines}, pattern color=u3, draw=u3, thick] (double-3-1.north west)--(double-3-3.north east)--(double-4-3.south east)--(double-4-1.south west)--(double-3-1.north west);
\end{scope}

\begin{scope}[on background layer]
\filldraw[preaction={fill, u4!20}, draw=u4, pattern = {horizontal lines}, pattern color = u4,  thick] (double-3-4.north west)--(double-3-5.north east)--(double-4-5.south east)--(double-4-4.south west)--(double-3-4.north west);
\end{scope}

\matrix(cell)[right=100pt of double, matrix of nodes, nodes in empty cells, nodes={circle, fill = black, scale=.5}, column sep = 30pt, row sep = 30pt]
{
 & & & & \\
 & & & & \\
 & & & & \\
 & & & & \\
};
\foreach \a in {1, ..., 4}{
\draw (cell-\a-1)-- (cell-\a-2) node(\a-12h) [midway, circle, fill = black, scale=.5]{};
\draw (cell-\a-2)-- (cell-\a-3) node(\a-23h) [midway, circle, fill = black, scale=.5]{};
\draw (cell-\a-3)-- (cell-\a-4) node(\a-34h) [midway, circle, fill = black, scale=.5]{};
\draw (cell-\a-4)-- (cell-\a-5) node(\a-45h) [midway, circle, fill = black, scale=.5]{};
}
\foreach \a in {1, ..., 5}{
\draw (cell-1-\a)-- (cell-2-\a) node(\a-12v) [midway, circle, fill = black, scale=.5]{};
\draw (cell-2-\a)-- (cell-3-\a) node(\a-23v) [midway, circle, fill = black, scale=.5]{};
\draw (cell-3-\a)-- (cell-4-\a) node(\a-34v) [midway, circle, fill = black, scale=.5]{};
}
\node[below right=1pt of cell-1-1]{$\cell_{a}$};
\node[below left=1pt of cell-1-5]{$\cell_{b}$};
\node[above right=1pt of cell-4-1]{$\cell_{c}$};
\node[above left=1pt of cell-4-5]{$\cell_{d}$};
\coordinate(midab) at ($(2-23h)!0.5!(3-23h)$);
\coordinate(midcd) at ($(4-23v)!0.5!(3-23v)$);

\begin{scope}[on background layer]
\fill[preaction={fill, u1!30}, pattern={vertical lines}, pattern color=u1] (cell-1-1.center)--(1-34h.center)--(2-23h.center)--(midab.center)--(2-23v.center)--(1-23v.center)--(cell-1-1.center);
\end{scope}

\begin{scope}[on background layer]
\fill[preaction={fill, u2!30},pattern = {north east lines}, pattern color = u2] (1-34h.center)--(cell-1-5.center)--(5-23v.center)--(midab.center)--(2-23h.center)--(1-34h.center);
\end{scope}

\begin{scope}[on background layer]
\fill[preaction={fill, u3!30}, pattern={north west lines}, pattern color=u3] (1-23v.center)--(midcd.center)--(3-34h.center)--(4-34h.center)--(cell-4-1.center)--(1-23v.center);
\end{scope}

\begin{scope}[on background layer]
\fill[preaction={fill, u4!30}, pattern = {horizontal lines}, pattern color = u4] (midcd.center)--(4-23v.center)--(5-23v.center)--(cell-4-5.center)--(4-34h.center)--(midcd.center);
\end{scope}
\draw [-{Stealth[length=5mm]},shorten >=15pt,shorten <=15pt,thick, color = redgentil] (base) -- (voron);
\draw [-{Stealth[length=5mm]},shorten >=15pt,shorten <=15pt,thick, color = redgentil] (voron) -- (double);
\draw [-{Stealth[length=5mm]},shorten >=15pt,shorten <=15pt,thick, color = redgentil] (double) -- (cell);
\end{tikzpicture}
\caption{Going from $\Sm_0$ to $\Sm$ and defining the cells.}
\label{fig:cells}
\end{figure}

Given a graph $\mathscr{G}$, define a new graph $\mathscr{G}'$ by performing a so-called \emph{edge subdivision} at each edge, i.e.~by adding a new vertex at the middle of each edge. Formally, one has $V_{\mathscr{G}'}=V_{\mathscr{G}}\sqcup E_{\mathscr{G}}$. As for the edges, for every $u\in V_{\mathscr{G}}$ and every edge $e\in E_\mathscr{G}$ containing $u$, we declare $u$ and $e$ to be adjacent in $\mathscr{G}'$, and no other edge is introduced. Set $\La=\La_0'$ and $\Sm=\Sm_0'$, and $\pi$ the map $V_{\La_0}\sqcup E_{\La_0}\to V_{\Sm_0}\sqcup E_{\Sm_0}$ induced by $\pi_0$.
Notice that, as $\centres$ is a subset of $V_{\Sm_0}$, it is also a subset of $V_\Sm$.
For $u\in \centres$, we define $\cell_u=\precell_u\sqcup \{e\in E_{\Sm_0}\,:\,e\cap \precell_u\neq \varnothing\}$. Observe the edge-boundary of $\precell_u$ in $\Sm_0$ is equal to the vertex-boundary of $\cell_u$ in $\Sm$. At last, we set $r=2r_0$ --- to take into account that distances are twice larger in $\Sm$ than in $\Sm_0$ --- and $R=4r_0+1$. Making all these choices indeed produces data fitting the setup we were targeting.

To conclude the reduction, it remains to explain why the inequality $p_c(\La)<p_c(\Sm)$ suffices to yield $p_c(\La_0)<p_c(\Sm_0)$, which is easy as, for any graph, we have $p_c(\mathscr{G})=p_c(\mathscr{G}')^2$.

\begin{rema}\label{rem:bond}
The only reason why we restrict ourselves to bond percolation goes as follows. In the remaining of the proof, we work in the aforementioned setup to prove the desired conclusion. The arguments will adapt readily to site percolation, provided we replace Item~\ref{item:partition} of the setup by the stronger condition ``the $\cell_u$'s form a partition of $V_\Sm$'' --- and they would still work if we further weaken Item~\ref{item:connected-interior} to ``for every $u\in\centres$ and every $v\in \partial_V\cell_u$, there is a path from $u$ to $v$ in $\cell_u$ that, apart from its final step, avoids $\partial_V\cell_u$''. It is however not clear how to reduce the general situation to this modified setup.
\end{rema}

\subsection{Step 2: Introducing the augmented percolation model on $\Sm$}

Let $p\in[0,1]$ and let $(X_e)_{e\in E_\Sm}$ be a collection of independent Bernoulli random variables of parameter $p$. An edge $e$ satisfying $X_e=1$ is said to be \defini{$p$-open}. Let $s\in[0,1]$ and let $(Y_\cell)$ be a collection of independent Bernoulli random variables of parameter $s$, indexed by the cells. The families $(X_e)$ and $(Y_\cell)$ are also taken to be independent of one another.

For $A\subset V_\Sm$, we define $\super(A)$, the augmented cluster of $A$, to be the smallest set $K\subset V_\Sm$ containing $A$ and satisfying the following ``closure'' properties:
\begin{enumerate}
    \item for every $v\in K$ and every neighbour $w$ of $v$, if $X_{\{v,w\}}=1$, then $w$ belongs to $K$,
    \item\label{item:s-closure} for every cell $\cell$, every $v\in \partial_V \cell$ and every $w\in\mathring{\cell}$ adjacent to $v$ , if $v\in K$ and $X_{\{v,w\}}=1$ and all edges $e\subset \mathring{\cell}$ satisfy $X_e=1$ and $Y_{\cell}=1$, then $K$ contains the whole cell $\cell$.
\end{enumerate}
The augmented cluster of a vertex $v$ is defined to be $\super(v):=\super(\{v\})$. Notice that if $u$ belongs to $\super(v)$, it does not imply that $v$ belongs to $\super(u)$ --- see Figure~\ref{fig:dissymetric}.

\begin{figure}[h]
    \centering

\scalebox{0.65}{\begin{tikzpicture}

\node(cellule)[circle,minimum size = 100pt, fill=gray!20, draw=black]{$\mathring\cell\text{ open, } Y_\cell = 1$};
\node(bord)[circle,draw = black, minimum size = 120pt]{};
\node(u)[left=0pt of bord]{$v$};
\node(upoint)[left=0pt of bord, anchor=center, circle,fill = black, scale=.5]{};
\node(vpoint)[above=0pt of bord, anchor=center, circle,fill = black, scale=.5]{};
\node(v)[above=0pt of vpoint]{$u$};
\draw[black, thick] (upoint.east) -- (cellule.west);
\node[below=10pt of bord]{$\cell$};
\node(center)[right = 150pt of cellule]{};

\node(Zu)[above = 50pt of center, circle,minimum size = 120pt, fill=redgentil!20, draw=redgentil]{};
\node(Zuu)[left=0pt of Zu]{$v$};
\node(Zuupoint)[left=0pt of Zu, anchor=center, circle, fill = redgentil, scale=.5]{};
\node(Zuvpoint)[above=0pt of Zu, anchor=center, circle, fill = redgentil, scale=.5]{};
\node(Zuv)[above=0pt of Zuvpoint]{$u$};
\node()[above right = 2pt of Zu, text = redgentil]{$K^+_{p, s}(v)$};

\node(Zv)[below = 50pt of center, circle,minimum size = 120pt, draw=black]{};
\node(Zvvpoint)[above=0pt of Zv, anchor=center, circle, fill = redgentil, scale=1]{};
\node(Zvupoint)[left=0pt of Zv, anchor=center, circle, fill = black, scale=.5]{};
\node(Zvu)[left=0pt of Zv]{$v$};
\node(Zvv)[above=0pt of Zvvpoint]{$u$};
\draw[black, thick] (upoint.east) -- (cellule.west);
\begin{scope}[on background layer]
\end{scope}
\node()[above right = 2pt of Zvvpoint, text = redgentil]{$K^+_{p, s}(u)$};

\draw [-{Stealth[length=5mm]},shorten >=75pt,shorten <=75pt,thick, color = redgentil] (cellule.center) -- (Zu.center) node[midway, above left]{$v \leftrightarrow \mathring\cell$};
\draw [-{Stealth[length=5mm]},shorten >=75pt,shorten <=75pt,thick, color = redgentil] (cellule.center) -- (Zv.center) node[midway, below left]{$u \nleftrightarrow \mathring\cell$};

\end{tikzpicture}}

    \caption{Assuming that the $X$'s and $Y$'s are 0 in the neighbouring cells, this is a situation where $u$ belongs to $\super(v)$ but $v$ does not belong to $\super(u)$.}
    \label{fig:dissymetric}
\end{figure}
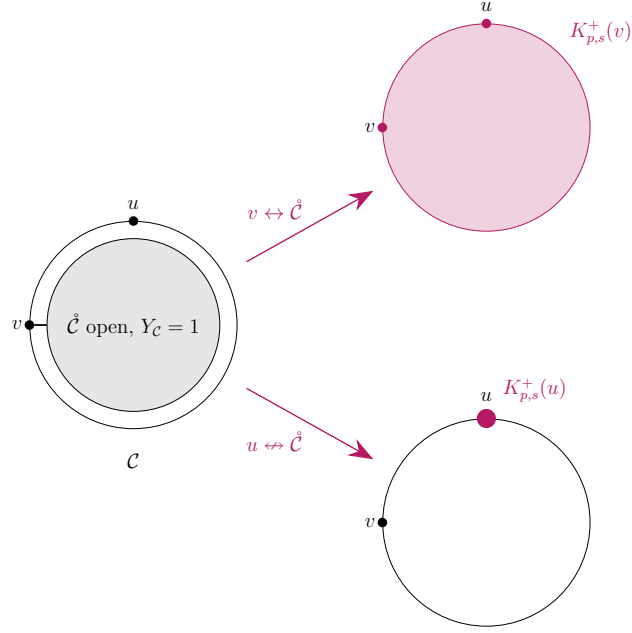

Given $s\in[0,1]$, there is a unique $\pcaug(\Sm,s,v)$ such that
\begin{itemize}
    \item for every $p<\pcaug(\Sm,s,v)$, we have $\Pp(\super(v)\text{ is infinite})=0$,
    \item for every $p>\pcaug(\Sm,s,v)$, we have $\Pp(\super(v)\text{ is infinite})>0$.
\end{itemize}
If $u$ and $v$ are in the same connected component of $\Sm$, then we have $\pcaug(\Sm,s,u)=\pcaug(\Sm,s,v)$. Therefore, if $\mathscr{C}$ denotes a connected component of $\Sm$, the quantity $\pcaug(\mathscr{C},s)$ is well defined. Besides, by Kolmogorov's zero-one law, the following conditions hold:
\begin{itemize}
    \item for every $p<\pcaug(\mathscr{C},s)$, almost surely, for every $v\in\mathscr{C}$, the augmented cluster of $v$ is finite,
    \item for every $p>\pcaug(\mathscr{C},s)$, there is almost surely a vertex in $\mathscr{C}$ such that its augmented cluster is infinite.
\end{itemize}
It makes sense to set $\pcaug(\Sm,s):=\inf_\mathscr{C} \pcaug(\mathscr{C},s)$, as we then have this dichotomy:
\begin{itemize}
    \item for every $p<\pcaug(\Sm,s)$, almost surely, for every $v\in V_\Sm$, the augmented cluster of $v$ is finite,
    \item for every $p>\pcaug(\Sm,s)$, there is almost surely a vertex in $\Sm$ such that its augmented cluster is infinite.
\end{itemize}

What will play the role of $p_c^\mathrm{new}(\Sm)$ from the heuristics of page~\pageref{strat} will be $\pcaug(\Sm,s)$, for small enough values of $s$. More precisely, in Step~3, we shall prove that for every $s\in (0,1]$, we have $\pcaug(\Sm,s)<p_c^\mathrm{bond}(\Sm)$. Then, we shall see in Step~4 that for $s$ small enough, we have $p_c^\mathrm{bond}(\La)\le \pcaug(\Sm,s)$. Step~4 is established by using a standard exploration algorithm, revealing edges at the boundary of the cluster one at a time. The novelty lies in Step~3, which is proved by using \emph{another} algorithm: there, we proceed cell by cell rather than edge by edge.

\subsection{Step 3: For every $s\in (0,1]$, we have $\pcaug(\Sm,s)<p_c^\mathrm{bond}(\Sm)$}

Let $s\in(0,1]$. Our purpose is to find some $v_0\in V_\Sm$ such that $\pcaug(\Sm,s,v_0)< p_c^\mathrm{bond}(\Sm)$.

Let $v_0$ be a vertex lying at the boundary of its cell. Then, one way to explore its $p$-cluster\footnote{namely its connected component for the graph defined by the $p$-open edges: this $p$-cluster is simply a cluster of bond percolation with parameter $p$} is by using the following algorithm. It does not really explore the full cluster, only what happens at the boundaries of the cells --- which is enough, as the cluster is infinite if and only if its intersection with $\bigcup_\cell \partial_V\cell$ is infinite. We will call an element of $\bigcup_\cell \partial_V\cell$ a \defini{boundary-vertex}.
\begin{enumerate}
    \item Before applying the algorithm, enumerate the cells: when we will say ``pick a cell such that \emph{something}'', it will mean that we pick the cell with smallest label among all cells satisfying \emph{something}. This will guarantee univoque definition and measurability of the process, conditional independance (we do not pick a cell depending on what is inside), and that any cell that could be picked will be picked at some point even if the algorithm goes on forever.\footnote{We could make our argument to work without taking care of this last point.}
    \item Initialise the process by declaring all cells to be unexplored, and set the current cluster to be $\{v_0\}$.
    \item \label{item:reveal-cell}Pick some unexplored cell $\cell$ such that its boundary contains a vertex of the current cluster. Reveal all $X_e$'s for $e$ in the cell $\cell$. We do not need to remember all these values, only the answer to each of the following questions: for any $v$ and $w$ in $\partial_V\cell$, the question is ``is there a $p$-open path connecting $v$ to $w$ inside $\cell$?''. When the answer is yes, we say that $v$ is $p$-connected to $w$ in $\cell$. Then, the cell $\cell$ becomes explored.
    \item \label{item:saturate}Add to the current cluster all boundary-vertices $v$ such that there are a boundary-vertex $w$ in the current cluster and an explored cell $\cell$ containing both $v$ and $w$ such that $v$ is $p$-connected to $w$ in $\cell$.
    \item \label{item:saturate2}As long as it adds new vertices, iterate Instruction~\ref{item:saturate}.
    \item As long as there are unexplored cells, iterate Instructions~\ref{item:reveal-cell}--\ref{item:saturate2}.
\end{enumerate}

The algorithm stops after finitely many steps if and only if the $p$-cluster of $v_0$ is finite. If it stops after finitely many steps, define the final cluster to be the current cluster at the end of the algorithm. If the algorithm goes on forever, the current cluster is an increasing sequence of sets: in this case, we can define the final cluster to be the union of these sets. Notice that the final cluster produced by the algorithm consists of the intersection of the true percolation cluster with the set of boundary-vertices.

How can this be helpful to us?  To explain how, let us introduce, for each cell $\cell$, the space $\bound_\cell$ of all equivalence relations on $\partial_V\cell$. Elements of $\bound_\cell$ will typically serve to capture equivalence relations such as ``being $p$-connected in $\cell$''. For each $\cell$, the finite set $\bound_\cell$ is endowed with a partial order as follows. We say that $\eta$ is smaller than $\eta'$ if, for any $v$ and $w$ in $\partial_V\cell$, whenever $v$ and $w$ are in the same $\eta$-class, then they are also in the same $\eta'$-class.

The strategy is to define an algorithm similar to the previous one exploring the boundary-vertices of the $(p,s)$-augmented percolation cluster of $v_0$. Due to the augmentation, we will be able to find some $q>p$ such that for every $\cell$, the random equivalence relation ``being $q$-connected in $\cell$'' is stochastically dominated by ``being $(p,s)$-connected in $\cell$''. Therefore, the algorithm will provide a coupling such that if the $q$-cluster of $v_0$ is infinite, then so is the $(p,s)$-augmented cluster. With some extra care regarding how to choose $p$ and $q$, we will be able to conclude Step~3 from there.

To define the exploration algorithm for $(p,s)$-augmented percolation, we only modify Instructions~3 and 4. Besides, we consider that we picked the same enumeration in Instruction~1 for both algorithms. Instructions~3 and 4 are as follows.
\begin{enumerate}\setcounter{enumi}{2}
\item \label{item:reveal-bonus}Pick some unexplored cell $\cell$ such that its boundary contains a vertex of the current cluster. Reveal all $X_e$'s for $e$ in the cell $\cell$, as well as $Y_\cell$. We do not need to remember all these values, only the answer to each of the following questions: for any $v$ and $w$ in $\partial_V\cell$, the question is ``is it the case that either there is a $p$-open path from $v$ to $w$ in $\cell$ or something else happens?''. In the previous sentence, ``something else'' refers to all the following conditions holding simultaneously: there is a $p$-open edge in $\cell$ such that one of its endpoints belongs to $\mathring\cell$ and the other one belongs to the current cluster; all edges lying in $\mathring \cell$ are $p$-open; $Y_\cell$ is equal to 1. When the answer to the question is yes, we say that $v$ is $(p,s)$-connected to $w$ in $\cell$.  Then, the cell $\cell$ becomes explored. {\it We insist that $v$ being $(p,s)$-connected to $w$ in $\cell$ is a Boolean that is defined now and will not be updated later. In other words, at any subsquent step, when we ask if vertices are $(p,s)$-connected in $\cell$, the ``current cluster'' appearing in the definition of ``$(p,s)$-connected'' is the cluster at the beginning of the instruction revealing $\cell$.}
\item Add to the current cluster all boundary-vertices $v$ such that there are a boundary-vertex $w$ in the current cluster and an explored cell $\cell$ containing both $v$ and $w$ such that $v$ is $(p,s)$-connected to $w$ in $\cell$.
\end{enumerate}

The final cluster produced by the algorithm is a subset\footnote{This is generally not an equality because, in the closure property~\ref{item:s-closure} of page~\pageref{item:s-closure}, it may well be that suitable $v$'s exist but that none of them is in the \emph{current} version of the cluster.} of the intersection of the true augmented cluster with the set of boundary-vertices. In particular, if the cluster produced by this algorithm is infinite, then $K^+_{p,s}(v_0)$ is infinite.

Denote by $Z_p(\cell)$ the random equivalence relation defined by ``being $p$-connected in $\cell$''. For $A$ a nonempty subset of $\partial_V\cell$, denote by $Z^A_{p,s}(\cell)$ the random equivalence relation defined as in Instruction~\ref{item:reveal-bonus}, but with $A$ instead of ``the current cluster''.
Because of the above algorithms, if the first of the following conditions holds, then so does so second:
\begin{itemize}
    \item for every $\cell$ and $A$, the distribution of $Z_q(\cell)$ is stochastically dominated by that of $Z^A_{p,s}(\cell)$,
    \item the intersection of the $q$-cluster of $v_0$ with the boundary-vertices has a distribution stochastically dominated by the same for the $(p,s)$-augmented cluster of $v_0$.
\end{itemize}

Combining this observation with the following lemma, we will be able to conclude.

\begin{lemm}
    \label{beurre:domicell}
    Let $\varepsilon>0$. There is some $\delta\in(0,\varepsilon]$ depending only on $R$, $s$, $\varepsilon$ and our upper bound on the maximal degree of $\Sm$ such that for every $p\in [\varepsilon,1-\varepsilon]$, for every $\cell$ and $A$, the distribution of $Z_{p+\delta}(\cell)$ is stochastically dominated by that of $Z^A_{p,s}(\cell)$.
\end{lemm}

\begin{rema}
As for Lemma~\ref{lem:one-column}, this lemma follows rather easily from Strassen's Theorem for monotone couplings. Here again, we opted for the Strassen-free proof because, despite being longer, it seemed to us it was more instructive.
\end{rema}

\begin{proof}
Once we know $R$ and an upper bound on the maximal degree of $\Sm$, there are, up to isomorphism, only finitely many possibilities for $(\cell,A)$. Therefore, it suffices, given $\varepsilon$, to find some $\delta$ depending on $(s,\varepsilon,\cell,A)$ such that for every $p\in [\varepsilon,1-\varepsilon]$, the distribution of $Z_{p+\delta}(\cell)$ is stochastically dominated by that of $Z^A_{p,s}(\cell)$. Let us prove this.

Let $\varepsilon>0$. Let $\cell$ be a cell and $A$ a nonempty finite subset of $\partial_V\cell$. Let $p\in [\varepsilon,1-\varepsilon]$. Let $\mu_p$ be the distribution of $(Z_p(\cell),Z^A_{p,s}(\cell))$. Notice that $\mu_p$ is a monotone coupling between its first and its second marginal. We use the notation $\minconfig$ for the minimum of $\bound_\cell$, namely the equivalence relation with singleton classes. We use the notation $\maxconfig$ for the maximum of $\bound_\cell$, namely the equivalence relation with a single class. Because $\mathring\cell$ is connected, if we have the following conditions:
\begin{itemize}
    \item edges lying in $\mathring\cell$ are all open,
    \item other edges in $\cell$ are closed  except for a single one, which besides has an endpoint in $A$,
    \item $Y_\cell=1$,
\end{itemize}
then $Z_p(\cell)$ is $\minconfig$ and $Z^+_{p,s}(\cell)$ is $\maxconfig$. Therefore the event ``$Z_p(\cell)=\minconfig$ and $Z^+_{p,s}(\cell)=\maxconfig$'' has probability at least $\alpha:=s\cdot\varepsilon^{|E_\cell|}>0$.

Recall that $\mu_p$ is a probability measure on $(\bound_\cell)^2$. 
Notice that, for every $\eta\in \bound_\cell$, the quantity $\nu_p(\eta):=\sum_{\eta'}\mu_p(\eta,\eta')$ depends polynomially, hence uniformly continuously, on $p\in[0,1]$. Therefore, we can pick $\delta\in (0,\varepsilon]$ such that for every $(p,p')\in[0,1]^2$, if $|p-p'|\le \delta$, then we have $\beta:=\sum_{\eta}|\nu_{p'}(\eta)-\nu_{p}(\eta)|\le\alpha$.  By taking $p\in[\varepsilon, 1-\varepsilon]$ and $p'=p+\delta$, this means that, considering probability measures as vectors, the $L^1$-distance between the first marginal of $\mu_p$ and $\nu_{p+\delta}$ is $\beta\le\alpha$.  In order to get a monotone coupling between $\nu_{p+\delta}$ and the distribution of $Z^A_{p,s}(\cell)$, we start from $\mu_p$ --- which has the desired second marginal --- and ``move mass around'' to adjust the first marginal.

If $\nu_{p+\delta}(\eta)>\nu_p(\eta)$, we want to take mass $\nu_{p+\delta}(\eta)-\nu_p(\eta)$ from $(\minconfig,\maxconfig)$ and reallocate it to $(\eta,\maxconfig)$: this transforms the weight of $\eta$ into the desired one without changing the second marginal and without breaking monotonicity of the coupling, since $\eta\le\maxconfig$. If $\nu_{p+\delta}(\eta)\le\nu_p(\eta)$, we want to take mass $\nu_{p}(\eta)-\nu_{p+\delta}(\eta)$ from one or several $(\eta,\eta')$ and reallocate it to $(\minconfig,\eta')$: this transforms the weight of $\eta$ into the desired one without changing the second marginal and without breaking monotonicity of the coupling, since $\eta'\ge\minconfig$. Doing so for all $\eta\neq \minconfig$, we get a monotone coupling with correct marginals except possibly for $\minconfig$ --- we do not run into negative weights precisely because\footnote{Actually, as the total increase of the weights matches its total decrease, the inequality $\beta\le2\alpha$ would suffice.} $\beta\le \alpha$. Then, because probability measures have total mass 1, the marginals must also be correct for $\minconfig$.
\end{proof}

We are now ready to conclude Step~3. Pick $\varepsilon$ such that $p_c(\Sm)\in (\varepsilon, 1-\varepsilon)$. Pick some $\delta$ satisfying the conclusion of Lemma~\ref{beurre:domicell}. We can pick $v_0$ such that $p_c^{v_0}(\Sm)<p_c(\Sm)+\delta$, where $p_c^{v_0}(\Sm)$ denotes $p_c^\mathrm{bond}$ of the connected component of $v_0$ in $\Sm$. We can thus pick $q$ such that $\varepsilon < p_c(\Sm)\le p_c^{v_0}(\Sm)<q<p_c(\Sm)+\delta$. Therefore, we can pick $p$ such that $p+\delta\ge q$ and $\varepsilon<p<p_c(\Sm)$ --- either $p=q-\delta$ or $p=\frac{p_c(\Sm)+\varepsilon}{2}$ does the job. As $p+\delta \ge q>p_c^{v_0}(\Sm)$, the $(p+\delta)$-cluster of $v_0$ is infinite with positive probability. Because of Lemma~\ref{beurre:domicell} and the observation made just before it, this means that the $(p,s)$-augmented cluster of $v_0$ is infinite with positive probability. This gives $\pcaug(\Sm,s)\le p<p_c(\Sm)$, thus ending Step~3.

\subsection{Step 4: For some $s\in (0,1]$, we have $p_c^\mathrm{bond}(\La)\le p_c^\mathrm{aug}(\Sm,s)$.}
\label{sec:last-step}

This part is performed exactly as in the proof of Proposition~2.1 from \cite{ms}, to which the reader is referred for details. It is even simpler in the present case as the sets $E_\cell$ are disjoint, so that there is no need for the ``multiple edges'' trick. Let us still provide a sketch of the argument.

Let $p\in(0,1]$. We explore a spanning tree of the $p$-cluster of $v_0$ and lift it to $\La$ in the usual way, revealing one by one the edges adjacent to the explored part of the cluster. But if our $p$-open explored edges fill $B_u(r)$ in $\Sm$, then, since we can switch floors by lifting cycles of length at most some constant $c$, we will be able to connect to some unexplored copy\footnote{formally an unexplored lift of a spanning tree of $\cell_u$} of  $\cell_u$ with probability at least $p^c>0$. Performing $p$-percolation on this copy, there is some probability at least $p^M>0$ that this copy turns out to be fully open --- where $M:=\mathrm{maxdeg}(\Sm)^{R}$ is an upper bound on the cardinality of the edge-set of a spanning tree of $\cell_u$. Setting $s_p:=p^{M+c}>0$, these observations yield a coupling such that, almost surely, the image by $\pi$ of the $p$-cluster of $v_0'$ contains the $(p,s_p)$-augmented cluster of $v_0$.

How to conclude from there? If $p_c^\mathrm{bond}(\La)=0$, then there is nothing to prove. Otherwise, setting $s=s_{p_c^\mathrm{bond}(\La)}$ yields the statement of Step~4, thus concluding the proof of Theorem~\ref{thm:strict}.\qed

\vspace{0.2cm}

In light of Remarks~\ref{rem:compawithms} and \ref{rem:bond}, we conclude by asking the following question.

\begin{enonce}{Question}
Is it possible to adapt our proof of Theorem~\ref{thm:strict} to cover the case of site percolation? 
\end{enonce}
\newpage
\bibliographystyle{smfalpha}
\bibliography{biblio}
\end{document}